\documentclass[11pt]{amsart}

\usepackage[T1]{fontenc}
\usepackage[utf8]{inputenc}
\usepackage{amsmath,amssymb,amsthm,mathtools}
\usepackage{mathrsfs}
\usepackage[rgb,dvipsnames]{xcolor}
\usepackage{tikz}
\usetikzlibrary{backgrounds}
\usetikzlibrary{arrows,arrows.meta,petri,topaths,positioning,shapes,shapes.misc,patterns,calc,decorations,decorations.pathreplacing,hobby}
\usepackage{tikz-cd}
\usepackage{adjustbox}
\usepackage{caption}
\usepackage{floatrow}
\usepackage{float}
\usepackage[inline]{enumitem}
\usepackage{tasks}
\usepackage{booktabs}
\usepackage{array}
\usepackage{microtype}
\usepackage{relsize}
\usepackage[hidelinks]{hyperref}
\usepackage[backend=biber,style=alphabetic,maxbibnames=15,maxcitenames=6,dateabbrev=false,autolang=hyphen]{biblatex}

\renewcommand{\UrlFont}{\sffamily\smaller}
\renewbibmacro{in:}{\ifentrytype{article}{}{\printtext{\bibstring{in}\intitlepunct}}}
\DeclareFieldFormat{url}{{\UrlFont\url{#1}}}
\DeclareFieldFormat{urldate}{%
  (version \thefield{urlday}\addspace%
  \mkbibmonth{\thefield{urlmonth}}\addspace%
  \thefield{urlyear}\isdot)}
\DeclareFieldFormat{eprint:arxiv}{%
\ifhyperref
    {\href{http://arxiv.org/abs/#1}{%
    arXiv\addcolon\nolinkurl{#1}}\iffieldundef{eprintclass}{}{\printtext{ [\thefield{eprintclass}]}}}
    {arXiv\addcolon\nolinkurl{#1}\iffieldundef{eprintclass}{}{\printtext{ [\thefield{eprintclass}]}}}}
\DeclareFieldFormat{eprint:arXiv}{%
\ifhyperref
    {\href{http://arxiv.org/abs/#1}{%
    arXiv\addcolon\nolinkurl{#1}}\iffieldundef{eprintclass}{}{\printtext{ [\thefield{eprintclass}]}}}
    {arXiv\addcolon\nolinkurl{#1}\iffieldundef{eprintclass}{}{\printtext{ [\thefield{eprintclass}]}}}}
\defbibheading{mlobibliography}{\begin{center}\scshape References\end{center}}

\newtheorem{theorem}{Theorem}[section]
\newtheorem{lemma}[theorem]{Lemma}
\newtheorem{corollary}[theorem]{Corollary}
\newtheorem{proposition}[theorem]{Proposition}

\newtheorem*{maintheorem*}{Main Theorem}
\theoremstyle{definition}
\newtheorem{definition}[theorem]{Definition}
\newtheorem*{maindefinition*}{Main Definition}

\theoremstyle{remark}
\newtheorem{remark}[theorem]{Remark}
\theoremstyle{plain}

\newsavebox{\mlopossiblebox}
\newsavebox{\mlonecessarybox}
\sbox{\mlopossiblebox}{\tikz[scale=.6ex/1cm,baseline=-.6ex,rotate=45,line width=.1ex]{\draw (-1,-1) rectangle (1,1);}}
\sbox{\mlonecessarybox}{\tikz[scale=.6ex/1cm,baseline=-.6ex,line width=.1ex]{\draw (-1,-1) rectangle (1,1);}}
\DeclareMathOperator{\possible}{\text{\usebox{\mlopossiblebox}}}
\DeclareMathOperator{\necessary}{\text{\usebox{\mlonecessarybox}}}
\newcommand{\satisfies}{\models}
\DeclareRobustCommand{\axiomf}[1]{\ifmmode\text{\upshape #1}\else\textup{#1}\fi}
\DeclareRobustCommand{\theoryf}[1]{\ifmmode\mathrm{#1}\else\textup{#1}\fi}

\newcommand{\SFour}{\ensuremath{\theoryf{S4}}}
\newcommand{\SFourTwo}{\ensuremath{\theoryf{S4.2}}}
\newcommand{\SFourTwoOne}{\ensuremath{\theoryf{S4.2.1}}}
\newcommand{\SFourThree}{\ensuremath{\theoryf{S4.3}}}
\newcommand{\SFive}{\ensuremath{\theoryf{S5}}}

\renewcommand{\emptyset}{\varnothing}

\newcommand{\Q}{\mathbb Q}
\newcommand{\Z}{\mathbb Z}
\newcommand{\Val}{\mathrm{Val}}
\newcommand{\Valsent}{\mathrm{Val}^{\mathrm{sent}}}
\newcommand{\Cone}{\operatorname{Cone}}
\newcommand{\diag}{\Delta_0}
\newcommand{\Th}{\operatorname{Th}}
\newcommand{\LOemb}{\mathrm{LO}^{\mathrm{emb}}}
\newcommand{\LOmon}{\mathrm{LO}^{\mathrm{mon}}}
\newcommand{\LOcond}{\mathrm{LO}^{\mathrm{cond}}}
\newcommand{\LOend}{\mathrm{LO}^{\mathrm{end}}}
\newcommand{\Lord}{\mathcal L_{\le}}
\newcommand{\Lmord}{\mathcal L^{\possible}_{\le}}
\newcommand{\LredA}[1]{\mathcal L_{\le}(#1)}
\DeclareRobustCommand{\Sfourthreecap}{\ifmmode\SFourThree_{\axiomf{cap}}\else\SFourThree\textsubscript{\axiomf{cap}}\fi}
\newcommand{\mainref}[1]{\textsuperscript{\normalfont\tiny\color{SalmonSunrise}\ref{#1}}}
\newcommand{\powerset}{\mathscr P}
\newcommand{\restr}{\upharpoonright}
\newcommand{\emb}{\hookrightarrow}
\newcommand{\onto}{\twoheadrightarrow}
\newcommand{\Iff}{\mathrel{\Longleftrightarrow}}

\definecolor{SalmonSunrise}{HTML}{F29783}
\definecolor{MintBreeze}{HTML}{83F2CA}
\definecolor{LemonChiffon}{HTML}{F2EDB3}
\definecolor{SkyBlueDream}{HTML}{A9CDF6}
\definecolor{LavenderHaze}{HTML}{A88FF2}

\tikzset{>=Stealth,
  dot/.style={circle,draw,fill,inner sep=#1},
    dot/.default=.7pt
}
\colorlet{opred}{Salmon!85}
\colorlet{opgreen}{Aquamarine!85}
\colorlet{opyellow}{Dandelion!45}
\colorlet{opblue}{CornflowerBlue!60}
\colorlet{oppurple}{Plum!75}

\title{The Modal Theory of Linear Orders}
\author{Wojciech Aleksander Wo\l oszyn}
\address[Wojciech Aleksander Wo\l oszyn]
{Mathematical Institute, University of Oxford, Andrew Wiles Building, Radcliffe Observatory Quarter, Woodstock Road, Oxford, OX2 6GG, United Kingdom \&\ St Hilda's College, Cowley Place, Oxford, OX4 1DY, United Kingdom}
\email{wojciech@woloszyn.org}
\urladdr{https://woloszyn.org}

\begin{document}

\begin{abstract}
I study the modal theory of linear orders under embeddings, monotone maps, condensations, and end-extensions. I prove modality elimination for embeddings and monotone maps, show that condensations make scatteredness modally definable, and compute exact propositional modal validities in the main cases.
\end{abstract}

\maketitle

Modal linear order theory investigates linear orders from a modal perspective. The worlds are always linear orders, but the allowed maps determine what it means for one order to become another: embeddings insert points; monotone maps collapse intervals; condensations produce ordered quotients; and end-extensions append final tails. In each case we ask what becomes expressible in the resulting modal language and exactly which propositional modal principles are valid, both with and without parameters.

\begin{maintheorem*}\makeatletter\def\@currentlabelname{main theorem}\makeatother\label{Maintheorem.LO}
\leavevmode
\begin{enumerate}[label=\textup{(\arabic*)},leftmargin=*,itemsep=0.05em,topsep=0pt,parsep=0pt]

\item In the category of linear orders and order-embeddings, the embedding modality admits modality elimination\mainref{thm:elimination-emb}. Every occurrence of $\possible\theta$ can be replaced by a finite disjunction of order-patterns together with upper bounds on the successive intervals they determine. Consequently, finite worlds have exact sentential modal logic \Sfourthreecap{}\mainref{thm:finite-emb-sentential}, while infinite worlds have exact sentential modal logic \SFive\mainref{thm:infinite-emb-sentential}. In the presence of parameters, a world has exact modal logic \SFourTwo\mainref{thm:emb-S42-iff-adjacent} precisely when it contains infinitely many adjacent pairs, and exact modal logic \SFive\mainref{thm:emb-S5-iff-DLO} precisely when it is a nonempty dense linear order without endpoints.

\item In the category of linear orders and monotone maps, the monotone-map modality likewise admits modality elimination\mainref{thm:elimination-mon}. For $\possible\theta$, satisfiable target order-patterns contribute the disjunction of all source patterns that collapse to them, with one extra empty-order alternative in the sentence case. Consequently, every nonempty world has exact sentential modal logic \SFive\mainref{thm:nonempty-mon-sentential}. With parameters allowed, singleton worlds have exact modal logic \SFive\mainref{thm:singleton-mon-formulaic}, while nonempty infinite worlds have exact modal logic \SFourTwo\mainref{thm:mon-S42-iff-infinite}.

\item In the category of nonempty linear orders and condensations, modality elimination fails. Indeed, scatteredness is itself modally definable\mainref{thm:scattered-modally-definable}. Every world validates \SFourTwoOne\mainref{thm:S421-cond-all} with parameters, and this logic is exact for every non-scattered world\mainref{thm:nonscattered-condensation-exact-S421}.

\item In the category of linear orders and end-extensions, every finite world validates exactly \SFour\mainref{thm:end-finite-S4}, even with parameters allowed, and every world admits an end-extension validating \SFive\mainref{thm:end-extends-to-S5} with parameters.

\end{enumerate}
\end{maintheorem*}
\newpage

The theorem shows that the same class of structures has markedly different modal behaviour depending on which maps are allowed. Under embeddings, finite size constraints act like buttons. Namely, once sufficiently many points have been inserted into the relevant intervals, those constraints are necessary; an infinite order has already realized all such finite possibilities at the sentential level. Under monotone maps, the sentential structure collapses even further. Every nonempty order maps constantly onto any singleton and, more generally, the internal order-pattern is invisible to sentential modal assertions, yielding exact \SFive. With parameters, however, infinite worlds recover enough room for the exact logic \SFourTwo. Condensations behave differently. Modality elimination fails, and the modal language can define scatteredness, a classical non-first-order property of linear orders. In the non-scattered case this expressivity is calibrated exactly by \SFourTwoOne. Finally, end-extensions are more rigid, since all new points must be added at the top; nevertheless finite worlds still validate precisely \SFour, while every world can be extended to one validating \SFive. For the exactness assertions above, exactness means that every modal principle outside the stated logic fails under some substitution of the indicated kind.

In this article, the order language is
$\Lord=\{\le\}$, with $x<y$ abbreviating $x\le y\wedge x\neq y$.  Linear orders
may be empty unless explicitly excluded; in the condensation category the worlds
are nonempty linear orders.

\vspace{0.8\baselineskip}

Meanwhile, the modal theory of linear orders forms part of the broader project
of modal model theory \cite{HW}, the study of structures inside classes of
similar structures equipped with a possibly refined substructure relation giving
rise to notions of possibility and necessity. Its antecedents include
Hamkins's Simple Maximality Principle \cite{HamkinsSimpleMaximality}, the modal
logic of forcing of Hamkins and L{\"o}we \cite{HamkinsLoweForcing}, and the
structural analysis of forcing classes by Hamkins, Leibman, and L{\"o}we
\cite{HamkinsLeibmanLowe}, where the button, switch, and ratchet methods were
systematically connected with modal validities. The subject grew further through
set-theoretic potentialism and the potentialist maximality principles
\cite{HamkinsLinnebo}, as well as arithmetic potentialism \cite{HamkinsArithmetic},
and has since expanded to a variety of structural settings, including the modal
logic of inner models \cite{InamdarLowe}, the modal logic of abelian groups
\cite{BergerBlockLowe}, and modal group theory \cite{WolGroups}. The modal
theory of the category of sets introduced the Kripke-category framework used
throughout this paper \cite{WolSets}, allowing one to study arbitrary concrete
categories of first-order structures rather than only traditional extension-based
potentialist systems. This broader perspective is essential for linear orders:
embeddings, monotone maps, condensations, and end-extensions all act on the same
underlying class of structures, yet they induce strikingly different modal
behaviours and different propositional validities.

\section*{Acknowledgements}
I should like to thank Joel David Hamkins and Ehud Hrushovski for helpful discussions.

\section{Modal semantics for a concrete category}\label{sec:semantics}

Let me first recall the formal semantics for the Kripke-category framework. A
Kripke category packages the two ingredients needed for the modal interpretation:
the structures themselves and the maps by which one is allowed to move from one
structure to another. Possibility and necessity are then interpreted inside the
cone of worlds reachable from the present one. For the standard finite-frame
facts used below we refer to \cite{BRV,CZ}.

\begin{maindefinition*}\makeatletter\def\@currentlabelname{main definition}\makeatother\label{Definition.Kripke-category}
A \emph{Kripke category} is a concrete category of $\mathcal L$-structures in a common
first-order language $\mathcal L$ such that every morphism is an $\mathcal L$-homomorphism.
Its objects are called \emph{worlds}, and its morphisms are called
\emph{accessibility mappings}.
\end{maindefinition*}

\begin{definition}[Cones and modal satisfaction]\label{def:modal-semantics}
Let $W$ be a world in a Kripke category $\mathcal K$. The \emph{cone above $W$},
denoted $\Cone(W)$, is the full subcategory of $\mathcal K$ whose objects are
the worlds accessible from $W$.

If $\mathcal L^{\possible}$ is the modal expansion of $\mathcal L$, the
satisfaction relation for formulas of $\mathcal L^{\possible}$ is defined
simultaneously at all worlds in $\Cone(W)$ by the usual first-order clauses
together with
\begin{align*}
U\satisfies \possible\varphi[\nu]
&\quad\text{if and only if there is }f:U\to V\text{ in }\Cone(W)\\
&\quad\text{such that }V\satisfies \varphi[f\circ \nu],\\
U\satisfies \necessary\varphi[\nu]
&\quad\text{if and only if for every }f:U\to V\text{ in }\Cone(W)\\
&\quad\text{we have }V\satisfies \varphi[f\circ \nu].
\end{align*}
\end{definition}

\begin{definition}[Parameter languages]\label{def:parameter-languages}
For a parameter set $A$ in a linear order $W$, let $\Lmord(A)$ denote the
modal expansion of $\Lord$ by constant symbols naming the elements of $A$, and
let $\LredA{A}$ denote its non-modal reduct. Thus $\Lmord=\Lmord(\varnothing)$.
\end{definition}

We shall use the following renaming lemma from \cite[Renaming lemma]{WolSets}.

\begin{proposition}[Renaming lemma]\label{prop:renaming-lemma}
If $\pi:W\cong U$ is an isomorphism in a Kripke category, then for every modal
formula $\varphi$ and every valuation $\nu$,
\[
W\satisfies \varphi[\nu]
\quad\text{if and only if}\quad
U\satisfies \varphi[\pi\circ\nu].
\]
\end{proposition}

\section{Propositional modal theory}\label{sec:propositional-modal-theory}

I will use propositional modal theories as sets of modal formulas, rather than
as proof systems. The distinction matters in modal model theory, because
validities with parameters need not be closed under necessitation in the usual
proof-theoretic sense. The substitution language will always be specified.

\begin{definition}[The modal theories used in this paper]\label{def:modal-theories}
We use the normal propositional modal theories $\SFour$, $\SFourTwo$,
$\SFourThree$, $\SFive$, and $\SFourTwoOne$. Here
\[
\SFourTwo=\SFour+\axiomf{.2},\qquad
\SFourThree=\SFour+\axiomf{.3},\qquad
\SFive=\SFour+\axiomf{5},
\]
and
\[
\SFourTwoOne:=\SFourTwo+\axiomf{.1}.
\]
The McKinsey axiom $\axiomf{.1}$ is
\[
\necessary\possible p\to \possible\necessary p,
\]
and we use $\axiomf{5}$ in the maximality-principle form
\[
\possible\necessary p\to p.
\]
Semantically, $\SFour$ reflects reflexivity and transitivity of the
accessibility relation; $\axiomf{.2}$ records directedness; $\axiomf{.3}$
records the absence of branching in generated cones; $\axiomf{.1}$ is the
McKinsey principle that a necessary possibility is possibly necessary; and
$\axiomf{5}$ says that any possibly necessary assertion is already true.
\end{definition}

\begin{definition}[Propositional validities]\label{def:propositional-validities}
Let $\mathcal K$ be a Kripke category of $\mathcal L$-structures, let $W$ be a
world of $\mathcal K$, and let $\Gamma$ be a collection of
$\mathcal L^{\possible}$-assertions. A propositional modal formula
$\psi(p_0,\dots,p_n)$ is \emph{valid at $W$ with respect to $\Gamma$} if every
substitution instance $\psi(\varphi_0,\dots,\varphi_n)$, with each
$\varphi_i\in\Gamma$, holds at $W$.

We write $\Val_{\mathcal K}(W,\Gamma)$ for the set of all such valid
propositional modal formulas. When only sentence members of $\Gamma$ are
allowed as substitutions, we write $\Valsent_{\mathcal K}(W,\Gamma)$ and call
these the \emph{sentential} validities at $W$. When arbitrary assertions from a
parameter language such as $\Lmord(A)$ are allowed, $\Val_{\mathcal K}(W,\Gamma)$
gives the corresponding \emph{formulaic} validities.
\end{definition}

Thus, an assertion such as $\Val_{\mathcal K}(W,\Gamma)=\SFourTwo$ says two
things at once: all principles of $\SFourTwo$ hold throughout the relevant cone,
and every propositional principle outside $\SFourTwo$ is refutable at $W$ by
substitutions from $\Gamma$.

We shall also use the cone lemma repeatedly \cite[Cone Lemma]{WolSets}.

\begin{proposition}[Cone lemma]\label{prop:cone-lemma}
Let $T$ be a propositional modal theory closed under modal propositional
substitution. The normal closure of $T$ is valid at a world $W$ with respect to a
substitution collection $\Gamma$ if and only if every formula of $T$ is true at
every world in $\Cone(W)$ under every substitution from $\Gamma$.
\end{proposition}

\section{Propositional modal validities in a Kripke category}\label{sec:modal-validities-Kripke-category}

The general theory of Kripke categories supplies the lower and upper-bound tools
used below. I will use only a small part of that machinery here: amalgamation for
the lower bound $\SFourTwo$, and the standard control-statement and labeling
methods for upper bounds.

\begin{definition}[Amalgamability over parameters]\label{def:A-amalgamability}
Let $A\subseteq W$ be a set of parameters in a world $W$. A span
\[
U_0\xleftarrow{f_0}W\xrightarrow{f_1}U_1
\]
is \emph{$A$-amalgamable} if there are arrows $g_0:U_0\to U$ and
$g_1:U_1\to U$ such that
\[
(g_0\circ f_0)(a)=(g_1\circ f_1)(a)
\]
for all $a\in A$.

The cone above $W$ is \emph{$A$-amalgamable} if, for every arrow $h:W\to V$,
every span rooted at $V$ is $h[A]$-amalgamable.
\end{definition}

The following is the only lower-bound theorem used in the paper; it is a special
case of the standard lower-bound theorem for Kripke categories
\cite[Section~4]{WolSets}.

\begin{theorem}[$\SFourTwo$ lower bound]\label{thm:S42-lower-bound}
Let $W$ be a world in a Kripke category $\mathcal K$, and let $A\subseteq W$.
If the cone above $W$ is $A$-amalgamable, then
\[
\SFourTwo\subseteq
\Val_{\mathcal K}(W,\mathcal L^{\possible}(A)).
\]
\end{theorem}

We also use a small amount of standard propositional frame theory for finite
reflexive transitive frames.

\begin{definition}[Finite frames and frame maps]\label{def:finite-frames}
A \emph{(propositional) Kripke frame} is a preorder $F=(X,\le_F)$. A
\emph{pointed frame} is a frame together with a designated root $w_0\in X$. A
\emph{Kripke model} on $F$ is a pair $(F,V)$, where $V$ assigns to each
propositional variable a subset of $X$.

For $A\subseteq X$, the \emph{generated subframe} $F_A$ is the restriction of
$F$ to the set
\[
X_A:=\{x\in X:\text{for some }a\in A,\ a\le_F x\}.
\]
In the pointed case we write $F_{w_0}:=F_{\{w_0\}}$.

A map $\pi:F\to G$ between preorders is a \emph{bounded morphism} if:
\begin{enumerate}[label=\textup{(\roman*)}]
    \item $u\le_F v$ implies $\pi(u)\le_G \pi(v)$;
    \item whenever $\pi(u)\le_G y$, there is $v\in F$ with $u\le_F v$ and
    $\pi(v)=y$.
\end{enumerate}

For a finite preorder $F$, write $x\sim_F y$ when $x\le_F y$ and $y\le_F x$.
The quotient $F/{\sim_F}$ carries the induced partial order, and the
$\sim_F$-classes are called \emph{clusters}. A finite frame is \emph{complete}
if it has a single cluster. A finite frame is a \emph{pre-Boolean algebra} if
$F/{\sim_F}$ is a finite Boolean algebra.

A finite \emph{topped pre-Boolean algebra} is a finite preorder with a unique
largest point $t$ such that the subframe obtained by deleting $t$ is a finite
pre-Boolean algebra; equivalently, it is a finite pre-Boolean algebra with one
additional singleton top point. Inamdar and L{\"o}we call these frames
\emph{finite inverted lollipops}. We shall use their finite-frame theorem that
these frames characterize $\theoryf{S4.2Top}$ \cite[Theorem~6]{InamdarLowe}.
Inamdar and L{\"o}we define this logic using the Top axiom. Over $\theoryf{T}$,
Top is equivalent to
$\possible\necessary p\vee\possible\necessary\neg p$, and by duality to the
McKinsey axiom $\necessary\possible p\to\possible\necessary p$. Thus this is the
logic denoted here by $\SFourTwoOne=\SFourTwo+\axiomf{.1}$.

A finite rooted \emph{tree} is a finite partial order with a least element,
called the root, such that the predecessors of each node are linearly ordered. A
finite \emph{pretree} is a finite preorder whose quotient by mutual accessibility
is a finite rooted tree. A finite pretree is \emph{regular of type $(k,m)$},
where $k,m\geq 1$, if its quotient tree has every non-leaf with exactly $k$
children and every cluster has size $m$. Equivalently, it is obtained from a
finite rooted $k$-ary tree by replacing each node with an $m$-cluster. A finite
preorder is \emph{rooted} if it has a distinguished point from which every point
is accessible. A \emph{capped chain} is a finite rooted preorder whose cluster
quotient is a finite linear order and whose non-maximal clusters are singletons.
\end{definition}

\begin{proposition}[Generated subframes and bounded morphisms]\label{prop:frame-invariance}
Generated subframes preserve modal truth at all worlds they contain. Bounded
morphisms preserve modal truth under pullback valuations: if $\pi:F\to G$ is a
bounded morphism and $\pi^*V$ is the pullback of a valuation $V$ on $G$, then
\[
(F,\pi^*V,u)\satisfies\varphi
\quad\text{if and only if}\quad
(G,V,\pi(u))\satisfies\varphi
\]
for every propositional modal formula $\varphi$ and every $u\in F$.
\end{proposition}

These are the standard generated-subframe and bounded-morphism facts; see
\cite[Chapter~2]{BRV} or \cite[Chapter~2]{CZ}.

\begin{definition}[Buttons, dials, and $\sigma$-control]\label{def:control-statements}
Let $W$ be a world. A formula $b$ is a \emph{weak button} at $W$ if
$W\satisfies\possible\necessary b$. It is \emph{unpushed} if
$W\not\satisfies\necessary b$. A \emph{button} is a formula satisfying
$\necessary\possible\necessary b$. A button is \emph{pure} if
$b\to\necessary b$ holds throughout the relevant cone. A finite family of
buttons is \emph{independent} if, necessarily, any chosen subfamily can be
pushed without pushing any of the remaining buttons.

A finite list $d_0,\dots,d_{m-1}$ is a \emph{dial} if, necessarily, exactly one
$d_i$ holds and every $d_i$ is possible.

Suppose $\sigma$ is an unpushed pure button at $W$. A \emph{$\sigma$-dial} is a
dial whose values can be changed while remaining below $\sigma$, that is, while
preserving $\neg\sigma$. A \emph{pure $\sigma$-button} is a pure button $b$ such
that below $\sigma$, whenever $b$ is unpushed it can be pushed without pushing
$\sigma$, and once $\sigma$ is pushed then $b$ is automatically true.
\end{definition}

\begin{proposition}[Weak buttons and purification]\label{prop:button-purification}
Let $W$ be a world in a Kripke category.
\begin{enumerate}[label=\textup{(\arabic*)}]
    \item If $\SFourTwo$ is valid at $W$ for a given substitution language, then
    every weak button expressible in that language is already a button at $W$.
    \item Any independent family of buttons can be purified by replacing each
    $b_i$ with $\necessary b_i$.
\end{enumerate}
\end{proposition}

\begin{proof}
For \textup{(1)}, from $\possible\necessary b$ one obtains
$\necessary\possible\necessary b$ using $\axiomf{4}$ and $\axiomf{.2}$. For
\textup{(2)}, the same pattern of worlds witnessing the independence of the
$b_i$'s witnesses the independence of the pure buttons $\necessary b_i$; pushing
$b_i$ is exactly making $\necessary b_i$ true. Compare \cite[Section~4]{WolSets}.
\end{proof}

\begin{lemma}[Regularizing finite pretrees]\label{lem:pretree-regularization}
Let $F$ be a finite pretree with initial node $w_0$. Then there are a finite
regular pretree $F^\ast$, an initial node $w_0^\ast\in F^\ast$, and a
root-preserving bounded morphism
\[
\pi:(F^\ast,w_0^\ast)\to(F,w_0).
\]
\end{lemma}

\begin{proof}
Let $T=F/{\equiv}$ be the quotient tree by mutual accessibility, rooted at
$r=[w_0]$, and let $C_t$ be the cluster corresponding to $t\in T$. Write
$\operatorname{Child}_T(t)$ for the set of children of $t$ in $T$. Set
\[
k:=\max\bigl(\{1\}\cup\{\lvert\operatorname{Child}_T(t)\rvert:t\in T
\text{ non-leaf}\}\bigr),
\qquad
m:=\max\{\lvert C_t\rvert:t\in T\}.
\]
Choose surjections $\rho_t:\{0,\dots,k-1\}\twoheadrightarrow
\operatorname{Child}_T(t)$ for non-leaf $t$, and surjections
$\sigma_t:\{0,\dots,m-1\}\twoheadrightarrow C_t$ for all $t\in T$.
Unfold $T$ into a finite rooted $k$-ary tree $T^\ast$: the root is
$\varnothing$ with $\pi_T(\varnothing)=r$, and if $s\in T^\ast$ has
$\pi_T(s)=t$ non-leaf, then for each $i<k$ add a child $s^\frown i$ and put
$\pi_T(s^\frown i)=\rho_t(i)$. Leaves acquire no children. The map
$\pi_T:T^\ast\to T$ is a root-preserving bounded morphism. Define
\[
F^\ast:=T^\ast\times\{0,\dots,m-1\},
\qquad
(s,i)\le_{F^\ast}(u,j)
\quad\text{if and only if}\quad
s\le_{T^\ast}u.
\]
Then $F^\ast$ is regular of type $(k,m)$. Choose $i_0<m$ with
$\sigma_r(i_0)=w_0$, set $w_0^\ast=(\varnothing,i_0)$, and define
$\pi(s,i):=\sigma_{\pi_T(s)}(i)$. The monotonicity and back conditions follow
from those of $\pi_T$ and the surjectivity of the maps $\sigma_t$.
\end{proof}

\begin{definition}[Labelings and railyards]\label{def:labelings-railyards}
Let $F$ be a finite frame with initial node $w_0$. A \emph{labeling of $F$ above
$W$} is an assignment $u\mapsto\Phi_u$ of assertions to the nodes of $F$ such
that:
\begin{enumerate}[label=\textup{(\arabic*)}]
    \item $W\satisfies\Phi_{w_0}$;
    \item every world in $\Cone(W)$ satisfies exactly one label;
    \item if a world in $\Cone(W)$ satisfies $\Phi_u$, then the possible labels
    above it are exactly the $\Phi_v$ with $u\le_F v$.
\end{enumerate}
When the finite frame is a pretree, such a labeling will be called a
\emph{railyard labeling}.
\end{definition}

The next result is the standard labeling lemma in the form used below
\cite[Section~4]{WolSets}.

\begin{proposition}[Labeling lemma]\label{prop:labeling-lemma}
Suppose a finite frame $F$ with initial node $w_0$ is labeled above $W$ by
assertions from a substitution language $\mathcal L$. Then every propositional
Kripke model on $F$ can be simulated at $W$ by substituting propositional
variables with Boolean combinations of the labels. Consequently, if every frame
in a class $\mathfrak F$ admits such a labeling above $W$, then any propositional
formula not valid in $\mathfrak F$ fails at $W$ under a suitable substitution.
\end{proposition}

We shall also use the usual Jankov--Fine formulas; see \cite[Chapter~6]{BRV} or
\cite[Section~2.5]{CZ}.

\begin{proposition}[Jankov--Fine witnesses]\label{prop:jankov-fine}
For every finite pointed reflexive transitive frame $(F,w_0)$ there is a
propositional formula $\chi_F$ which fails at $(F,w_0)$ under the identity
valuation and whose failure at a world $W$ yields a labeling of $F$ above $W$.
\end{proposition}

\begin{proposition}[Upper-bounds theorem]
\label{prop:upper-bounds}
Let $W$ be a world in a Kripke category $\mathcal K$, and let $\mathcal L$ be a
nonempty substitution language closed under Boolean connectives.
\begin{enumerate}[label=\textup{(\arabic*)}]
    \item If $W$ admits arbitrarily long finite dials whose values lie in
    $\mathcal L$, then
    \[
        \Val_{\mathcal K}(W,\mathcal L)\subseteq \SFive.
    \]

    \item If $W$ admits arbitrarily large finite independent families of
    unpushed pure buttons together with arbitrarily long finite dials, independent
    from those buttons, all lying in $\mathcal L$, then
    \[
        \Val_{\mathcal K}(W,\mathcal L)\subseteq \SFourTwo.
    \]
    Moreover, if this control-statement hypothesis fails, then there is a
    single propositional modal formula $\chi\notin\SFourTwo$ such that
    \[
        \chi\in \Val_{\mathcal K}(W,\mathcal L).
    \]

    \item If every finite pretree admits a railyard labeling above $W$ using
    assertions from $\mathcal L$, then
    \[
        \Val_{\mathcal K}(W,\mathcal L)\subseteq \SFour.
    \]
    For this clause it is enough to handle only finite pretrees all of whose
    clusters have the same size and all of whose branching clusters have the
    same degree.

    \item If $W$ has an unpushed pure button $\sigma$ and, below $\sigma$,
    arbitrarily long finite $\sigma$-dials together with arbitrarily large
    finite independent families of unpushed pure $\sigma$-buttons, these being
    independent below $\sigma$, and all belonging to $\mathcal L$, then
    \[
        \Val_{\mathcal K}(W,\mathcal L)\subseteq \SFourTwoOne.
    \]
\end{enumerate}
\end{proposition}

\begin{proof}
Clauses~\textup{(1)}, the upper-bound implication in~\textup{(2)}, and the
pretree-labeling implication in~\textup{(3)} are the standard control-statement
upper bounds for Kripke categories; see \cite[Section~4]{WolSets} together with
\cite{CZ}. We record only the extra arguments used later.

For later reference, when $n,m\ge 1$ we write
\[
F_{n,m}:=\powerset(n)\times m,
\qquad
(S,i)\le (T,j)
\quad\text{if and only if}\quad
S\subseteq T,
\]
with initial node $(\varnothing,0)$. These canonical finite pre-Boolean
algebras are sufficient for the finite-frame upper-bound argument. Every finite
pre-Boolean algebra is a bounded-morphic image of some $F_{n,m}$, obtained by
choosing $m$ large enough to map onto each cluster. If $b_0,\dots,b_{n-1}$ are
pure buttons and $d_0,\dots,d_{m-1}$ is an independent $m$-dial, then for
$(S,i)\in F_{n,m}$ set
\[
\Phi_{S,i}:=d_i\wedge \bigwedge_{k\in S} b_k\wedge
\bigwedge_{k\notin S}\neg b_k.
\]
Purity makes the Boolean coordinate persistent under further accessibility,
while the dial records the cluster coordinate.

For the sharpness clause in \textup{(2)}, suppose the displayed
control-statement hypothesis fails. If every canonical frame $F_{n,m}$ were
labelable over $W$, then Boolean combinations of the labels $\Phi_{S,i}$ would
recover, for arbitrary $n$ and $m$, an $m$-dial and $n$ independent pure
buttons.
Hence some $F_{n,m}$ is not labelable. Proposition~\ref{prop:jankov-fine} then yields
a witness
$\chi_{F_{n,m}}\in\Val_{\mathcal K}(W,\mathcal L)$, while
$F_{n,m}\satisfies\SFourTwo$ and $F_{n,m}\not\satisfies\chi_{F_{n,m}}$; so
$\chi_{F_{n,m}}\notin\SFourTwo$.

For the reduction in \textup{(3)}, we use the standard fact that finite pretrees
form a complete finite frame class for $\SFour$. Every formula outside
$\SFour$ fails at the initial node of some finite pretree; see, for
example, \cite[Chapter~3]{BRV} or \cite[Chapter~5]{CZ}. Thus, if
$\psi\notin\SFour$, choose a finite pretree $F$, a node $w_0\in F$, and a
valuation $V$ with
$(F,V,w_0)\not\satisfies\psi$. Passing to the generated subframe $F_{w_0}$
preserves the refutation by generated-subframe preservation, so we may
assume that $w_0$ is initial. By Lemma~\ref{lem:pretree-regularization}, there are a
finite regular pretree $F^\ast$, an initial node $w_0^\ast$, and a
root-preserving bounded morphism $\pi:(F^\ast,w_0^\ast)\to(F,w_0)$. Pulling back
$V$ along $\pi$ and applying bounded-morphism invariance, we obtain
$(F^\ast,\pi^\ast V,w_0^\ast)\not\satisfies\psi$. A railyard labeling of $F^\ast$
above $W$ using assertions from $\mathcal L$ now transfers this refutation to
$W$ via Proposition~\ref{prop:labeling-lemma}. Hence
$\Val_{\mathcal K}(W,\mathcal L)\subseteq\SFour$.

For \textup{(4)}, the frame-theoretic input is the theorem of Inamdar and
L{\"o}we. It is enough to label finite topped pre-Boolean algebras. The
canonical topped pre-Boolean algebras suffice for this purpose. Namely, let
$E_{n,m}$ be obtained from $F_{n,m}$ by adding a singleton top node $t$ above all
points. Every finite topped pre-Boolean algebra is a bounded-morphic image of
some $E_{n,m}$, by choosing $n$ to match the number of atoms in the Boolean
quotient, choosing $m$ large enough to map onto every cluster below the top
point, and sending the top point to the top point. We therefore only need to
describe the labeling of $E_{n,m}$. Fix an $m$-value $\sigma$-dial
$d_0,\dots,d_{m-1}$ below $\sigma$
and an independent family $b_0,\dots,b_{n-1}$ of unpushed pure
$\sigma$-buttons. Label $t$ by $\sigma$, and label each $(S,i)\in F_{n,m}$ by
\[
\neg\sigma\wedge d_i\wedge \bigwedge_{k\in S} b_k\wedge
\bigwedge_{k\notin S}\neg b_k.
\]
Purity of $\sigma$ keeps the top node persistent, while below $\sigma$ the dial
controls the cluster coordinate and the pure $\sigma$-buttons control the
Boolean coordinate. Thus the accessibility relation is exactly that of
$E_{n,m}$, and Proposition~\ref{prop:labeling-lemma} yields
$\Val_{\mathcal K}(W,\mathcal L)\subseteq\SFourTwoOne$.
\end{proof}

\par\smallskip\noindent
We shall also use the following simple way of producing worlds with $\SFive$ modal theory.

\begin{theorem}
\label{thm:terminal-S5}
Let $W$ be a world in a Kripke category $\mathcal K$.
\begin{enumerate}[label=\textup{(\arabic*)}]
    \item If every world of $\mathcal K$ maps to $W$, then $W$ validates
    $\SFive$ for sentential substitutions.
    \item If in addition every map into $W$ is unique, then $W$ validates
    $\SFive$ for substitutions with parameters from $W$.
\end{enumerate}
\end{theorem}

\begin{proof}
We verify the $\axiomf{5}$ scheme on the whole cone above $W$, in the form
$\possible\necessary\vartheta\to\vartheta$.  Let $U$ be accessible from $W$.  By
hypothesis there is an arrow $U\to W$, and since $W\to U$ as well, every two
worlds in the cone above $W$ are mutually accessible by going through $W$.
Thus, if $U\satisfies\possible\necessary\vartheta$, choose $f:U\to V$ with
$V\satisfies\necessary\vartheta$.  The composite $V\to W\to U$ is an arrow from $V$ back
to $U$, so the boxed sentence at $V$ gives $U\satisfies\vartheta$.  This proves
clause \textup{(1)}.

For clause \textup{(2)}, let $h:W\to U$ be the arrow transporting the parameters
from $W$ to $U$, and suppose
$U\satisfies\possible\necessary\psi[h(\bar a)]$.  Choose $f:U\to V$ with
$V\satisfies\necessary\psi[f h(\bar a)]$.  Let $g:V\to W$ be the unique arrow into $W$.
Then $gfh:W\to W$ is an arrow into $W$, and hence $gfh=\mathrm{id}_W$.  Therefore
the arrow $hg:V\to U$ sends the parameters $f h(\bar a)$ back to $h(\bar a)$,
and the boxed formula at $V$ yields $U\satisfies\psi[h(\bar a)]$.  Since
$\SFour$ is valid in every Kripke category, both claims follow.
\end{proof}

\par\smallskip\noindent

Let $\kappa$ be an ordinal and let $\mathcal K$ be a Kripke category.
A \emph{$\kappa$-chain} in $\mathcal K$ is a coherent system
\[
\langle W_\alpha,f_{\alpha\beta}:W_\alpha\to W_\beta\rangle_{
\alpha\le\beta<\kappa}
\]
with $f_{\alpha\alpha}=\mathrm{id}_{W_\alpha}$ and
$f_{\beta\gamma}\circ f_{\alpha\beta}=f_{\alpha\gamma}$ whenever
$\alpha\le\beta\le\gamma<\kappa$.
A \emph{weak upper bound} of this chain is a world $N$ together with arrows
$h_\alpha:W_\alpha\to N$ for all $\alpha<\kappa$, with no commutativity
requirement.
A \emph{cocone} on the chain is a world $N$ together with arrows
$g_\alpha:W_\alpha\to N$ such that
$g_\beta\circ f_{\alpha\beta}=g_\alpha$ for all $\alpha\le\beta<\kappa$.
Such a cocone is \emph{covering} if every element of $N$ has the form
$g_\alpha(a)$ for some $\alpha<\kappa$ and some $a\in W_\alpha$.
\par\smallskip

\begin{theorem}
\label{prop:chains-to-S5}
Let $\mathcal K$ be a Kripke category.
\begin{enumerate}[label=\textup{(\arabic*)}]
    \item If every countable chain in $\mathcal K$ has a weak upper bound, then
    every world accesses one validating $\SFive$ with respect to any fixed
    countable family of sentential substitution instances.

    \item If every set-sized chain in $\mathcal K$ admits a cocone, then every
    world accesses one validating $\SFive$ with respect to sentences in any
    fixed set-sized language.

    \item If every set-sized chain in $\mathcal K$ admits a cocone and every
    countable chain in $\mathcal K$ admits a covering cocone, then every world
    accesses one validating $\SFive$ with respect to any fixed set-sized
    language, allowing parameters.
\end{enumerate}
\end{theorem}

\begin{proof}
Clauses~\textup{(1)} and~\textup{(2)} are the standard fusion construction. One
enumerates the relevant sentences and, at each stage, forces a possible
necessary sentence to become necessary; a weak upper bound or cocone for the
resulting chain then satisfies $\possible\necessary\varphi\to\necessary\varphi$ for the chosen
language.

For clause~\textup{(3)}, fix a set-sized language $\mathcal L$ of modal
formulas.  We use a two-level construction.  First fix a world $W$ and a
parameter set $A\subseteq W$.  Expand the language by constants
for the elements of $A$ and apply clause~\textup{(2)} to the set of closed
formulas $\necessary\psi[\bar a]$, where $\psi\in\mathcal L$ and
$\bar a\in A^{<\omega}$.  This gives an arrow $e:W\to V$ such that, for every
such $\psi$ and $\bar a$,
\[
V\satisfies
\possible\necessary\psi[e(\bar a)]\to\necessary\psi[e(\bar a)].
\]
Indeed, clause~\textup{(2)} gives the implication with
$\possible\necessary\necessary\psi[e(\bar a)]$ on the left, and
$\axiomf{4}$ gives
$\possible\necessary\psi[e(\bar a)]\to
\possible\necessary\necessary\psi[e(\bar a)]$.

Starting from the given world $W_0$, iterate this inner closure step countably
many times.  Thus, at stage $n$, apply the preceding paragraph with
$A=W_n$ and obtain an arrow $f_{n,n+1}:W_n\to W_{n+1}$ such that
\[
W_{n+1}\satisfies
\possible\necessary\psi[f_{n,n+1}(\bar a)]\to
\necessary\psi[f_{n,n+1}(\bar a)]
\tag{$\dagger$}
\]
for every $\psi\in\mathcal L$ and every finite tuple $\bar a$ from $W_n$.
Let $(g_n:W_n\to W_\omega)_{n<\omega}$ be a covering cocone for the resulting
countable chain.

We claim that $W_\omega$ validates the $\axiomf{5}$ scheme with parameters from
$W_\omega$.  Let $\bar b$ be a finite tuple from $W_\omega$.  By covering, choose
$n<\omega$ and a tuple $\bar a$ from $W_n$ such that $g_n(\bar a)=\bar b$, and set
$\bar a'=f_{n,n+1}(\bar a)$.  By commutativity of the cocone,
$g_{n+1}(\bar a')=\bar b$.  Suppose
$W_\omega\satisfies\possible\necessary\psi[\bar b]$, witnessed by an arrow
$h:W_\omega\to E$ with $E\satisfies\necessary\psi[h(\bar b)]$.  Then
$h\circ g_{n+1}:W_{n+1}\to E$ witnesses
$W_{n+1}\satisfies\possible\necessary\psi[\bar a']$.  By $(\dagger)$,
$W_{n+1}\satisfies\necessary\psi[\bar a']$.  If $r:W_\omega\to E'$ is any further
extension of $W_\omega$, then $r\circ g_{n+1}:W_{n+1}\to E'$ is an extension of
$W_{n+1}$, so $E'\satisfies\psi[r(g_{n+1}(\bar a'))]=\psi[r(\bar b)]$.
Thus $W_\omega\satisfies\necessary\psi[\bar b]$.

The same argument gives the scheme throughout the cone above $W_\omega$ for
parameters transported from $W_\omega$.  If $s:W_\omega\to U$ and
$U\satisfies\possible\necessary\psi[s(\bar b)]$, then already
$W_\omega\satisfies\possible\necessary\psi[\bar b]$, hence
$W_\omega\satisfies\necessary\psi[\bar b]$, and consequently
$U\satisfies\necessary\psi[s(\bar b)]$.  Therefore the normal theory
$\SFive$ is valid at $W_\omega$ in the parameter language.
\end{proof}

\par\smallskip\noindent

We also use the standard connection, from modal model theory, between
existential closedness and the maximality principle. A model $M$ of a theory
$T$ is \emph{existentially closed} if every existential formula with parameters
from $M$ that is realized in an extension model of $T$ is already realized in
$M$. When a category of models and embeddings admits modality elimination, a
world validates $\SFive$ in its full modal language with parameters
allowed exactly when it is existentially closed among the models of the
underlying theory; see \cite{HW}.

\section{Modality elimination for linear orders}

Let me begin with modality elimination for the two categories in which it holds:
the category of linear orders and order-embeddings, and the category of linear
orders and monotone maps.  Here modality elimination means that every modal
formula is equivalent, in the relevant category, to a modality-free formula.
Throughout this section, linear orders are allowed to be empty.

An \emph{ordered partition} of a finite set $X$ is a partition of $X$ into
blocks together with a linear order of those blocks.
If $P=(B_0<\cdots <B_{r-1})$ is an ordered partition of
$\{0,\dots,n-1\}$, we regard $P$ as the formula
\[
P(\bar x)\;:=\;
\bigwedge_{m<r}\ \bigwedge_{\substack{i,j<n\\ i,j\in B_m}} x_i=x_j
\ \wedge\ \bigwedge_{m<m'<r}\ \bigwedge_{\substack{i\in B_m\\ j\in B_{m'}}}x_i<x_j.
\]
Ordered partitions record exactly the quantifier-free information about a
finite tuple in a linear order: which coordinates are equal, and in what
left-to-right order the equality classes occur. The contrast between embeddings
and monotone maps is already visible here. Under embeddings, the equality
pattern of a tuple is preserved; under monotone maps, adjacent blocks may
collapse. For example, in any nonempty order the formula $\possible(x=y)$ is
true under monotone maps, since a constant map collapses $x$ and $y$, but under
embeddings it is equivalent to $x=y$. The next figure records several ordered
partitions of a $4$-element set. The dots are the variables, the colored regions
indicate equality classes, and the left-to-right order of the regions is the
order in which the corresponding values must appear in the ambient linear
order.

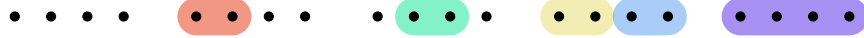
\begin{figure}[H]
\begin{equation*}
\begin{tikzpicture}[scale=.6,every node/.style={transform shape},
        Line/.style = {
                line width=4mm,
                line cap=round
            },
        Area/.style = {
                line width=5mm,
                fill,
                line cap=round,
                rounded corners=0.5mm
            }
    ]
    \newcommand{\FourPoints}{
        \foreach \k in {1,...,4}{
            \coordinate (N\k) at (0.8*\k,0);
            \draw[fill] (N\k) circle (1mm);
        }
    }

    \begin{scope}
        \FourPoints
    \end{scope}

    \begin{scope}[shift={(4,0)}]
        \FourPoints
        \begin{scope}[on background layer]
            \draw[Area,SalmonSunrise] (N1) -- (N2);
        \end{scope}
    \end{scope}

    \begin{scope}[shift={(8,0)}]
        \FourPoints
        \begin{scope}[on background layer]
            \draw[Area,MintBreeze] (N2) -- (N3);
        \end{scope}
    \end{scope}

    \begin{scope}[shift={(12,0)}]
        \FourPoints
        \begin{scope}[on background layer]
            \begin{scope}[blend group=multiply]
                \draw[Area,LemonChiffon] (N1) -- (N2);
                \draw[Area,SkyBlueDream] (N3) -- (N4);
            \end{scope}
        \end{scope}
    \end{scope}

    \begin{scope}[shift={(16,0)}]
        \FourPoints
        \begin{scope}[on background layer]
            \draw[Area,LavenderHaze] (N1) -- (N4);
        \end{scope}
    \end{scope}
\end{tikzpicture}
\end{equation*}
\caption{Examples of ordered partitions of a $4$-element set as satisfiable
patterns of equalities and strict inequalities in a linear order: dots are
variables, monochromatic regions indicate equalities, and the left-to-right
order of regions encodes the required $<$-relations.}
\label{fig:ordered-partitions}
\end{figure}

There is also a natural coarseness order on ordered partitions.  Given ordered
partitions $P$ and $Q$, we write $P\preccurlyeq Q$ if $Q$ is obtained from $P$
by merging adjacent blocks.  Thus $Q$ remembers less equality-and-order
information than $P$.  Figure~\ref{fig:ordered-partition-coarseness} illustrates
this relation.

\begin{figure}[H]
\begin{equation*}
\begin{tikzpicture}[scale=.6,every node/.style={transform shape},
        Line/.style = {
                line width=4mm,
                line cap=round
            },
        Area/.style = {
                line width=5mm,
                fill,
                line cap=round,
                rounded corners=0.5mm
            }
    ]

    \newcommand{\RowPoints}{
        \foreach \k in {1,...,4}{
            \coordinate (N\k) at (0.8*\k,0);
            \draw[fill] (N\k) circle (1mm);
        }
    }

    \begin{scope}
        \RowPoints
        \begin{scope}[on background layer]
            \draw[Area,SalmonSunrise] (N1) -- (N2);
        \end{scope}
    \end{scope}

    \begin{scope}[shift={(4.2,0)},scale=1.66666666667]
        \node at (0,0) {$\preccurlyeq$};
    \end{scope}

    \begin{scope}[shift={(4.8,0)}]
        \RowPoints
        \begin{scope}[on background layer]
            \draw[Area,MintBreeze, line width=5.5mm] (N1) -- (N3);
        \end{scope}
    \end{scope}

    \begin{scope}[shift={(9.2,0)}, scale=1.66666666667]
        \node at (0,0) {$\Iff$};
    \end{scope}

    \begin{scope}[shift={(9.8,0)}]
        \RowPoints
        \begin{scope}[on background layer]
            \draw[Area,MintBreeze, line width=6mm] (N1) -- (N3);
            \draw[Area,SalmonSunrise] (N1) -- (N2);
        \end{scope}
    \end{scope}

\end{tikzpicture}
\end{equation*}
\caption{Coarseness for ordered partitions: $P\preccurlyeq Q$ if and only if
$Q$ is obtained by merging adjacent blocks of $P$.}
\label{fig:ordered-partition-coarseness}
\end{figure}
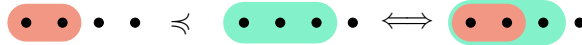

This coarseness relation is the reason ordered partitions are useful for
monotone maps.  In the category of linear orders and monotone maps, and likewise
for condensations, $P\preccurlyeq Q$ says exactly that a tuple realizing $P$ can
be sent by an accessibility mapping to a tuple realizing $Q$; the map may
identify adjacent equality blocks, but it cannot reverse their order.  In the
embedding category no such merging is possible, and an ordered partition is
preserved exactly.

Every quantifier-free $\Lord$-formula is therefore controlled by finitely many
ordered partitions.  We record this elementary reduction separately because it is
the point at which the later modality-elimination proofs become finite.

\begin{lemma}
\label{lem:qf-ordered-partitions}
Every satisfiable quantifier-free $\Lord$-formula is equivalent over the
theory of linear orders to a finite disjunction of ordered partitions of its
free variables.
\end{lemma}

\begin{proof}
Put the formula in disjunctive normal form and complete each satisfiable
conjunction of atomic and negated atomic formulas to a full order-pattern.
\end{proof}

\par\smallskip\noindent

We next name the intervals whose sizes control the embedding case.  If a tuple
is strictly increasing, then any embedding extension can add points only into
the gaps before the first coordinate, between consecutive coordinates, and after
the last coordinate.  The formulas below merely express finite upper and lower
bounds on those gaps.

Fix $k\in\omega$.
When $k=0$, let $I_0(\bar x)$ denote the whole order.
When $k>0$ and $x_0<\cdots<x_{k-1}$, let
\[
\begin{aligned}
I_0(\bar x)&=(-\infty,x_0),\\
I_j(\bar x)&=(x_{j-1},x_j)\qquad(1\le j\le k-1),\\
I_k(\bar x)&=(x_{k-1},\infty).
\end{aligned}
\]
If $\rho(z,\bar x)$ is any formula defining one of these intervals and
$n\in\omega$, write $\#\rho(\bar x)\ge n$ for the first-order formula
\[
\exists z_0,\dots,z_{n-1}\left(
\bigwedge_{i<n}\rho(z_i,\bar x)\ \wedge\
\bigwedge_{i<j<n} z_i\neq z_j
\right),
\]
and write $\#\rho(\bar x)\le n$ for
\[
\forall z_0,\dots,z_n\left(
\bigwedge_{i\le n}\rho(z_i,\bar x)\ \rightarrow\
\bigvee_{i<j\le n} z_i=z_j
\right).
\]
Finally, $\#\rho(\bar x)=n$ abbreviates the conjunction of
$\#\rho(\bar x)\ge n$ and $\#\rho(\bar x)\le n$.
For the consecutive intervals above we write these formulas as
$\#I_j(\bar x)\ge n$, $\#I_j(\bar x)\le n$, and $\#I_j(\bar x)=n$.
\par\smallskip

Let $\LOemb$ denote the Kripke category of linear orders and order-embeddings.
The point of the embedding case is that an embedding cannot change the
order-pattern of the named tuple.  All it can do is insert new points into the
successive intervals cut out by that tuple.  Thus the only data which the
possibility operator can add are upper bounds on how many points those intervals
can contain.

\begin{theorem}[Modality elimination for embeddings]
\label{thm:elimination-emb}
The category $\LOemb$ admits modality elimination.
More precisely, for every modal $\Lord$-formula $\theta(\bar x)$, the formula
$\possible\theta(\bar x)$ is equivalent in $\LOemb$ to a modality-free
$\Lord$-formula.
After fixing an ordered partition of the free variables, one may choose this
equivalent to be a finite disjunction of upper bounds on the sizes of the
consecutive intervals determined by that ordered partition.
\end{theorem}

\begin{proof}
We argue by induction on the complexity of modal formulas.  Atomic
formulas are already modality-free, and Boolean connectives and first-order
quantifiers preserve modality-freeness once it has been obtained for the
immediate subformulas.  Thus the only nontrivial case is a formula of the form
$\possible\psi(\bar x)$; all of the work below is devoted to understanding what
an embedding extension can add around the tuple $\bar x$.

By the induction hypothesis, every modal operator occurring strictly inside
$\psi$ can be eliminated.
So, replacing $\psi$ by an equivalent formula if necessary, we may assume that
$\psi(\bar x)$ is already a modality-free $\Lord$-formula.
Now distribute over the finitely many ordered partitions of the free
variables:
\[
\possible\psi(\bar x)\ \leftrightarrow\
\bigvee_P \possible\bigl(P(\bar x)\wedge\psi(\bar x)\bigr).
\]
If $P$ identifies two coordinates, merge them; after reindexing the remaining
free variables, it is enough to analyze either the sentence case $k=0$, or a
formula of the form
\[
(x_0<\cdots<x_{k-1})\wedge \possible\psi(\bar x)
\qquad (k>0).
\]
The argument below treats both cases simultaneously, with the interval-size notation just introduced, interpreting $I_0(\bar x)$ as the
whole order when $k=0$.

Let $E=\omega\cup\{\infty\}$, ordered by the usual order on $\omega$ together
with $n<\infty$ for all $n<\omega$.
For a linear order $M$ and an increasing $k$-tuple $\bar a$ in $M$
(or the empty tuple when $k=0$), let
$v(M,\bar a)\in E^{k+1}$ be the vector of the sizes of the consecutive
intervals $I_0(\bar a),\dots,I_k(\bar a)$.
Define
\[
S=\bigl\{v(M,\bar a): M \text{ is a linear order and }
M\satisfies (a_0<\cdots<a_{k-1})\wedge \psi(\bar a)\bigr\},
\]
where for $k=0$ the displayed order-pattern is interpreted as the empty
conjunction.
If $S=\varnothing$, then
$(x_0<\cdots<x_{k-1})\wedge\possible\psi(\bar x)$ is equivalent to false, so
there is nothing to prove.

We claim that $S$ is closed under coordinatewise suprema of nondecreasing
sequences.
Suppose
$s^0\le s^1\le\cdots$
is a nondecreasing sequence from $S$, and let
$b=\sup_i s^i$ coordinatewise.
Consider the first-order theory consisting of the theory of linear orders, the
formula
$(x_0<\cdots<x_{k-1})\wedge\psi(\bar x)$,
the exact cardinality statements
$\#I_j(\bar x)=b_j$
for each coordinate with $b_j<\infty$, and the lower bounds
$\#I_j(\bar x)\ge t$
for every $t<\omega$ whenever $b_j=\infty$.
Because $\psi$ is now modality-free, this is an ordinary first-order theory.
Any finite fragment mentions only finitely many lower bounds.
For each coordinate with finite supremum, the corresponding nondecreasing
sequence eventually stabilizes at that value; for each coordinate with
supremum $\infty$, some sufficiently large $s^i$ meets all the finitely many
required lower bounds.
Hence every finite fragment is realized by one of the tuples witnessing some
$s^i\in S$.
By compactness, the whole theory is satisfiable, so $b\in S$.

The poset $E$ is a well-quasi-order, and therefore so is the finite product
$E^{k+1}$.
We next show that every element of $S$ lies below
some maximal element of $S$.
Fix $s\in S$, and enumerate $E^{k+1}$ as
$e_0,e_1,e_2,\ldots$.
Starting with $s^0=s$, define a nondecreasing sequence in $S$ as follows.
Given $s^n$, if there is some $u\in S$ with $s^n\le u$ and $e_n\le u$, let
$s^{n+1}$ be the first such $u$ in the fixed enumeration; otherwise let
$s^{n+1}=s^n$.
Let $b=\sup_n s^n$ coordinatewise.
By the sequential-closure claim, $b\in S$.

We claim that $b$ is maximal in $S$.
Suppose $c\in S$ and $b\le c$.
Write $c=e_m$.
At stage $m$, the element $c$ itself witnesses that there is some $u\in S$ with
$s^m\le u$ and $e_m\le u$.
Hence $s^{m+1}$ was chosen so that $c=e_m\le s^{m+1}$.
Since $s^{m+1}\le b$, we have $c\le b$.
Together with $b\le c$, this gives $c=b$.
Thus $b$ is maximal.
Therefore every element of $S$ lies below a maximal element of $S$.
The set of maximal elements of $S$ is an antichain in a well-quasi-order, so
it is finite.
Write
\[
\max(S)=\bigl\{b^{(1)},\dots,b^{(r)}\bigr\}.
\]
For each $1\le i\le r$, define
\[
\beta_i(\bar x):=
\bigwedge_{\substack{j\le k\\ b^{(i)}_j<\infty}}
\#I_j(\bar x)\le b^{(i)}_j.
\]
We claim that in $\LOemb$,
\[
(x_0<\cdots<x_{k-1})\wedge \possible\psi(\bar x)
\ \leftrightarrow\
(x_0<\cdots<x_{k-1})\wedge
\bigl(\beta_1(\bar x)\vee\cdots\vee\beta_r(\bar x)\bigr),
\]
again interpreting the displayed order-pattern as vacuous when $k=0$.

For the forward direction, suppose
$W\satisfies (a_0<\cdots<a_{k-1})\wedge\possible\psi(\bar a)$.
Choose an embedding
$e:W\emb M$
with
$M\satisfies (e(a_0)<\cdots<e(a_{k-1}))\wedge\psi(e(\bar a))$.
Embeddings preserve order and are injective, so they cannot decrease any of
the consecutive interval sizes.
Hence
$v(W,\bar a)\le v(M,e(\bar a))$
coordinatewise.
Since
$v(M,e(\bar a))\in S$,
it lies below some maximal element
$b^{(i)}$,
and therefore
$W\satisfies \beta_i(\bar a)$.

For the converse, suppose
$W\satisfies a_0<\cdots<a_{k-1}$
and
$v(W,\bar a)\le b^{(i)}$
for some $i$.
Choose a linear order $M$ and a tuple
$\bar c=(c_0,\dots,c_{k-1})$
(or the empty tuple if $k=0$) such that
\[
M\satisfies (c_0<\cdots<c_{k-1})\wedge \psi(\bar c)
\]
and
$v(M,\bar c)=b^{(i)}$.
Work in a language with:
constant symbols naming all elements of $M\setminus\{\bar c\}$,
constant symbols naming all elements of $W\setminus\{\bar a\}$,
and shared constants
$d_0,\dots,d_{k-1}$
interpreted as
$c_0,\dots,c_{k-1}$
in $M$ and as
$a_0,\dots,a_{k-1}$
in $W$.
Let $T'$ consist of the elementary diagram of the expansion of $M$ by the
shared constants $d_0,\dots,d_{k-1}$ together with the full atomic diagram of the
expansion of $W$ by those same shared constants (including negated atomic
formulas, and in particular all distinctness statements).

Any finite fragment of $T'$ mentions only finitely many constants from
$W\setminus\{\bar a\}$.
Each such constant lies in one of the consecutive intervals determined by
$\bar a$.
If the corresponding coordinate of $b^{(i)}$ is finite, then that interval of
$M$ has exactly $b^{(i)}_j$ points, so the finitely many constants from $W$
can be interpreted injectively there because
$v(W,\bar a)\le b^{(i)}$.
If the coordinate is $\infty$, then the corresponding interval of $M$ is
infinite, and any finite linear order embeds into it.
Thus every finite fragment of $T'$ is satisfiable.
By compactness there is a model $M^*$ of $T'$.
Because $T'$ contains the elementary diagram of $(M,\bar c)$,
the reduct of $M^*$ to the symbols coming from $M$ is an elementary extension
of $M$, and so
$M^*\satisfies (d_0<\cdots<d_{k-1})\wedge\psi(\bar d)$.
Because $T'$ also contains the full atomic diagram of $(W,\bar a)$, the
interpretation of the constants from $W$ yields an embedding of $W$ into
$M^*$ sending $\bar a$ to $\bar d$.
Hence
$W\satisfies \possible\psi(\bar a)$.

This proves the displayed equivalence for formulas of the form
$(x_0<\cdots<x_{k-1})\wedge\possible\psi(\bar x)$, and therefore completes the
induction.
\end{proof}

\par
We next turn from embeddings to monotone maps.  Let $\LOmon$ denote the Kripke category of linear orders and monotone maps.

\begin{theorem}[Modality elimination for monotone maps]
\label{thm:elimination-mon}
The category $\LOmon$ admits modality elimination.
More precisely, for every modal $\Lord$-formula $\theta(\bar x)$, the formula
$\possible\theta(\bar x)$ is equivalent in $\LOmon$ to a modality-free
$\Lord$-formula.
For a nonempty tuple of free variables, after decomposing into ordered
partitions, each satisfiable target pattern $Q$ contributes precisely the finite
disjunction of all source patterns $P$ with $P\preccurlyeq Q$.
In the special case of no free variables, the additional normal form
``the order is empty'' may also occur.
\end{theorem}

\begin{proof}
We again argue by induction on the complexity of modal formulas.
The Boolean connectives and first-order quantifiers are routine, so it
suffices to consider a formula of the form
$\possible\psi(\bar x)$.
By the induction hypothesis, every modal operator occurring strictly inside
$\psi$ can be eliminated.
So we may assume that $\psi(\bar x)$ is modality-free.

First suppose that $\bar x$ is empty.
If $\psi$ is unsatisfiable in linear orders, then $\possible\psi$ is false
everywhere.
If $\psi$ has a nonempty model $M$, then every world satisfies
$\possible\psi$: the empty order maps uniquely to $M$, and every nonempty
order maps to $M$ by a constant monotone map.
Finally, if $\psi$ is satisfiable but only in the empty order, then
$\possible\psi$ holds exactly at the empty world, because no nonempty order
admits a function into the empty order.
This gives the asserted sentential normal forms.

Now assume that $\bar x$ is nonempty.
Decompose $\psi$ over the finitely many ordered partitions $Q$ of the free
variables:
\[
\possible\psi(\bar x)\ \leftrightarrow\
\bigvee_Q \possible\bigl(Q(\bar x)\wedge\psi(\bar x)\bigr).
\]
It remains to analyze one target ordered partition $Q$.
Let the blocks of $Q$ be
\[
B_0<\cdots<B_{k-1}
\]
with $k\ge1$.
Let $\psi_Q(y_0,\dots,y_{k-1})$ be obtained from $\psi$ by replacing every
variable in the block $B_i$ by $y_i$.
If
\[
(y_0<\cdots<y_{k-1})\wedge\psi_Q(\bar y)
\]
is unsatisfiable over the theory of linear orders, where the displayed
strict-order condition is vacuous when $k=1$, then
$\possible(Q\wedge\psi)$ is false everywhere.
Suppose instead that it is satisfiable, and choose a linear order $M$ and
points
\[
 c_0<\cdots<c_{k-1}
\]
with
\[
M\satisfies\psi_Q(c_0,\dots,c_{k-1}).
\]
We claim that for every world $W$ and every tuple $\bar a\in W$,
\[
W\satisfies\possible\bigl(Q(\bar a)\wedge\psi(\bar a)\bigr)
\qquad\text{if and only if}\qquad
W\satisfies \bigvee_{P\preccurlyeq Q} P(\bar a),
\]
where $P$ ranges over the ordered partitions of the original free variables.

For the forward implication, let $f:W\to N$ be a monotone map such that
$N\satisfies Q(f(\bar a))\wedge\psi(f(\bar a))$.
Let $P$ be the ordered partition realized by $\bar a$ in $W$.
A monotone map can only merge adjacent blocks of the source pattern, never
reverse their order, and the image pattern is $Q$.
Thus $P\preccurlyeq Q$.

Conversely, suppose that $\bar a$ realizes an ordered partition
$P\preccurlyeq Q$.
Let the blocks of $P$ be
\[
A_0<\cdots<A_{\ell-1},
\]
and write $b_s$ for the common value in $W$ of the variables belonging to
$A_s$.
Since $Q$ is obtained from $P$ by merging adjacent blocks, for each
$i<k$ the block $B_i$ is the union of a nonempty consecutive interval of the
$A_s$'s.
Let $r(i)$ be the largest index $s$ such that $A_s\subseteq B_i$.
Define convex subsets of $W$ by
\[
\begin{aligned}
J_0&=\{z\in W:z\le b_{r(0)}\},\\
J_i&=\{z\in W:b_{r(i-1)}<z\le b_{r(i)}\}\qquad(1\le i<k),\\
J_k&=\{z\in W:b_{r(k-1)}<z\}.
\end{aligned}
\]
Now define $f:W\to M$ by sending every point of $J_i$ to $c_i$ for
$i<k$, and every point of the final tail $J_k$ to $c_{k-1}$.
The sets $J_0,\dots,J_k$ occur in increasing order, and
$c_0<\cdots<c_{k-1}$, so $f$ is monotone.
Moreover, if a variable lies in the block $B_i$ of $Q$, then its value in $W$
lies in $J_i$, and hence its image under $f$ is $c_i$.
Therefore
\[
M\satisfies Q(f(\bar a))\wedge\psi(f(\bar a)),
\]
and so $W\satisfies\possible(Q\wedge\psi)(\bar a)$.

Thus every target partition contributes either nothing or the finite
disjunction of its refinements. Combining the finitely many target partitions
with the sentence case completes the induction.
\end{proof}

\section{The failure of modality elimination for condensations}

The preceding elimination theorems might suggest that the modal theory of
linear orders is always reducible to finite order-pattern information.
Condensations show that this is not so.  This is the first place where the
modal language sees a genuinely non-first-order property of linear orders.
A \emph{condensation} is a surjective monotone map between nonempty linear
orders.
We write $\LOcond$ for the Kripke category whose objects are the nonempty
linear orders and whose arrows are condensations.
A linear order is \emph{scattered} if it does not contain a copy of
$\Q$.

\begin{theorem}
\label{thm:scattered-modally-definable}
Scatteredness is definable in modal order theory for condensations.
More precisely, there is a modal sentence $\mathrm{Scatt}$ such that for every
nonempty linear order $L$,
\[
L\satisfies \mathrm{Scatt}
\qquad\text{if and only if}\qquad
L\text{ is scattered}.
\]
\end{theorem}

\begin{proof}
Let $\mathrm{Dense}$ be the first-order sentence asserting that the order is
dense and has both a minimum and a maximum:
\[
\exists \ell\exists r\Bigl(\ell<r\wedge
\forall x\,(\ell\le x)\wedge\forall x\,(x\le r)
\wedge\forall x\forall y\bigl(x<y\to\exists z\,(x<z\wedge z<y)\bigr)\Bigr).
\]
Define
\[
\mathrm{Scatt}:=\neg\possible\mathrm{Dense}.
\]
We use Doets's characterization that, assuming choice, a linear order is
scattered if and only if it admits no nontrivial densely ordered condensation
\cite[Lemma~3.2.3]{Doets}.

Suppose first that $L$ is scattered.
If $L\satisfies\possible\mathrm{Dense}$, choose a condensation
$f:L\onto D$ with $D\satisfies\mathrm{Dense}$.
Then $D$ is nontrivially densely ordered, contrary to Doets's characterization.
Hence $L\satisfies\mathrm{Scatt}$.

Conversely, suppose that $L$ is not scattered.
By Doets's characterization there is a condensation $f:L\onto D$ such that
$D$ is nontrivially densely ordered.
Choose $u<v$ in $D$, let $I=(u,v)_D$, and let $q,p\notin D$ be new points.
Form the ordered sum
\[
D' = \{q\}+I+\{p\}.
\]
Let $\iota:I\emb D'$ be the inclusion of the middle summand, and define
$g:D\to D'$ by
\[
g(b)=
\begin{cases}
 q,&b\le u,\\
 \iota(b),&u<b<v,\\
 p,&v\le b.
\end{cases}
\]
Then $g$ is monotone and surjective, hence a condensation.
Since $I$ is dense without endpoints, the order $D'$ is dense, with minimum
$q$ and maximum $p$; therefore $D'\satisfies\mathrm{Dense}$.
Thus $g\circ f:L\onto D'$ witnesses $L\satisfies\possible\mathrm{Dense}$, and so
$L\not\satisfies\mathrm{Scatt}$.
\end{proof}

\begin{corollary}
\label{cor:scattered-below}
There is a modal formula $\mathrm{Scatt}_{<}(x)$ such that for every nonempty
linear order $L$ and every $a\in L$,
\[
L\satisfies \mathrm{Scatt}_{<}[a]
\qquad\text{if and only if}\qquad
(-\infty,a)_L\text{ is scattered}.
\]
\end{corollary}

\begin{proof}
Let $\mathrm{Max}(x)$ express that $x$ is the maximum element, and define
\[
\mathrm{Scatt}_{<}(x):=\neg\possible\bigl(\mathrm{Max}(x)\wedge
\neg\mathrm{Scatt}\bigr).
\]
Fix a nonempty linear order $L$ and $a\in L$, and write
\[
L_{\le a}=(-\infty,a]_L
\]
for the closed initial segment through $a$ with its induced order.
First note that $(-\infty,a)_L$ is scattered if and only if $L_{\le a}$ is
scattered, since any copy of $\Q$ in $L_{\le a}$ cannot contain the
maximum element $a$.

Assume first that $(-\infty,a)_L$ is scattered, and hence that $L_{\le a}$ is
scattered.
Suppose towards a contradiction that
$L\satisfies\possible(\mathrm{Max}(a)\wedge\neg\mathrm{Scatt})$.
Choose a condensation $f:L\onto M$ such that, with $m_0=f(a)$,
\[
M\satisfies \mathrm{Max}[m_0]\wedge\neg\mathrm{Scatt}.
\]
If $a\le b$ in $L$, then monotonicity gives $m_0=f(a)\le f(b)$, while the
maximality of $m_0$ gives $f(b)\le m_0$; hence $f(b)=m_0$.
It follows that the restriction
\[
f\restr L_{\le a}:L_{\le a}\onto M
\]
is still monotone and surjective, hence a condensation.
Since $M\satisfies\neg\mathrm{Scatt}$, Theorem~\ref{thm:scattered-modally-definable}
gives a condensation $g:M\onto D$ with $D\satisfies\mathrm{Dense}$.
Then
\[
g\circ (f\restr L_{\le a}):L_{\le a}\onto D
\]
witnesses $L_{\le a}\satisfies\possible\mathrm{Dense}$, contradicting
Theorem~\ref{thm:scattered-modally-definable}.
Therefore $L\satisfies\mathrm{Scatt}_{<}[a]$.

Conversely, suppose that $(-\infty,a)_L$ is not scattered, and hence that
$L_{\le a}$ is not scattered.
Define the truncation map $\tau:L\to L_{\le a}$ by
\[
\tau(x)=
\begin{cases}
 x,&x\le a,\\
 a,&a<x.
\end{cases}
\]
Then $\tau$ is monotone and surjective, hence a condensation.
Clearly $L_{\le a}\satisfies \mathrm{Max}[\tau(a)]$, and since $L_{\le a}$ is not
scattered we also have $L_{\le a}\satisfies\neg\mathrm{Scatt}$.
Thus $\tau$ witnesses
\[
L\satisfies\possible\bigl(\mathrm{Max}(a)\wedge\neg\mathrm{Scatt}\bigr),
\]
so $L\not\satisfies\mathrm{Scatt}_{<}[a]$.
\end{proof}

\begin{corollary}
\label{cor:no-elim-cond}
The category $\LOcond$ does not admit modality elimination.
Indeed, the modal sentence $\mathrm{Scatt}$ is not equivalent in $\LOcond$ to
any first-order $\Lord$-sentence.
\end{corollary}

\begin{proof}
Suppose toward a contradiction that $\mathrm{Scatt}$ were equivalent to some
$\Lord$-sentence $\psi$.
Let $L_0=(\Z,\le)$ and
$L_1=(\Q\times \Z,\le_{\mathrm{lex}})$.
Then $L_0$ is scattered and $L_1$ is not.
On the other hand, both are discrete linear orders without endpoints, and so
both satisfy the complete first-order theory of endless discrete linear orders;
see \cite[Proposition~2.4.10]{Marker}.
Hence $L_0\equiv L_1$ in $\Lord$, contradicting the fact that they disagree on
$\mathrm{Scatt}$.
\end{proof}

\section{The modal theory of linear orders and embeddings}

Let us now turn to propositional validities. Embeddings are the most rigid of
the linear-order categories. They do not identify points, and so all modal
information comes from the possibility of inserting new points into cuts and
intervals. The elimination theorem from the previous section reduces modal
assertions to finite interval data, and the remaining task is to determine which
control statements those data support at a given world. Adjacent pairs are the
basic source of buttons. Splitting an adjacent pair is possible, and once split it
can never become adjacent again.

For $n\in\omega$, let $\theta_n$ denote the sentence ``there are at least
$n$ points'', with $\theta_0$ a tautology.

\begin{theorem}
\label{thm:S42-emb-all}
Every world in $\LOemb$ validates $\SFourTwo$ for arbitrary substitution
instances, even with parameters.
\end{theorem}

\begin{proof}
By Theorem~\ref{thm:S42-lower-bound} it is enough to show that every span in
$\LOemb$ is amalgamable.
Let
$L_0\xleftarrow{}L\xrightarrow{}L_1$
be a span of embeddings.
Replacing $L_0$ and $L_1$ by isomorphic copies, we may assume that
$L=L_0\cap L_1$.
Now consider the theory
\[
\Th(\mathrm{LO})\cup \diag(L_0)\cup \diag(L_1)
\]
in the language obtained by adding constant symbols for the elements of
$L_0\cup L_1$, where $\diag(A)$ denotes the full atomic diagram of $A$
(including negated atomic formulas, and in particular all distinctness
statements).
Any finite fragment merely specifies a finite order-pattern, and such a pattern
can obviously be realized in a linear order by inserting the finitely many new
constants into the cuts determined over $L$.
So the theory is finitely satisfiable and therefore satisfiable by compactness.
The resulting model is an amalgam.
\end{proof}

\begin{lemma}
\label{lem:upward-absolute-sentences}
Let $W$ be a world in $\LOemb$, let $\mathcal L$ be an intermediate language
\[
\Lord\subseteq \mathcal L\subseteq \Lmord,
\]
and let $b$ be a parameter-free sentence of $\mathcal L$.
Assume that $b$ is upward absolute on $\Cone(W)$, in the sense that for all
embeddings $U\emb V$ above $W$, if $U\satisfies b$, then $V\satisfies b$.
Then exactly one of the following holds on $\Cone(W)$:
\begin{enumerate}[label=\textup{(\roman*)}]
\item $b$ is false at every world in $\Cone(W)$;
\item there is $n\in\omega$ such that for every $U\in\Cone(W)$,
\[
U\satisfies b
\qquad\text{if and only if}\qquad
U\satisfies \theta_n.
\]
\end{enumerate}
In particular, any two such sentences are comparable by implication on the
cone above $W$. If some world above $W$ satisfies $b$, then case \textup{(i)}
is impossible, and therefore case \textup{(ii)} holds.
\end{lemma}

\begin{proof}
By Theorem~\ref{thm:elimination-emb}, the sentence $b$ is equivalent in
$\LOemb$ to a modality-free $\Lord$-sentence $\beta$.

We first claim that on $\Cone(W)$, $\neg b$ holds if and only if $\possible\neg b$ holds.
The forward implication is by reflexivity. For the reverse implication,
suppose $U\in\Cone(W)$ and $U\satisfies\possible\neg b$. Choose an embedding
$U\emb V$ with $V\in\Cone(W)$ and $V\satisfies\neg b$. If $U\satisfies b$, then the
assumed upward absoluteness of $b$ on $\Cone(W)$ would imply $V\satisfies b$, a
contradiction. Hence $U\satisfies\neg b$, establishing the claim.

Since $b\leftrightarrow\beta$ holds in $\LOemb$, we also have
$\neg b\leftrightarrow\neg\beta$ and therefore
$\possible\neg b\leftrightarrow\possible\neg\beta$ in $\LOemb$.
Combining this with the claim, we obtain on $\Cone(W)$ that $\neg b$ holds if and only if $\possible\neg\beta$ holds.
Now apply the sentence case $k=0$ of
Theorem~\ref{thm:elimination-emb} to the modality-free sentence $\neg\beta$.
When $k=0$, the only consecutive interval is the whole order, so
$\possible\neg\beta$ is equivalent in $\LOemb$ either to false, or to a finite
disjunction of upper bounds on the size of the whole order, possibly with one
vacuous disjunct coming from the maximal value $\infty$.
Equivalently, $\possible\neg\beta$ is one of the following:
false, true, or $\neg\theta_{m+1}$ for some $m\in\omega$.
If $\possible\neg\beta$ is false, then $\neg b$ is false throughout $\Cone(W)$,
so $b\leftrightarrow\theta_0$ on that cone.
If $\possible\neg\beta$ is true, then $\neg b$ holds throughout $\Cone(W)$, so
$b$ is false everywhere on that cone.
Finally, if
\[
\LOemb\satisfies \possible\neg\beta\leftrightarrow \neg\theta_{m+1},
\]
then on $\Cone(W)$ we have
\[
\neg b
\qquad\text{if and only if}\qquad
\neg\theta_{m+1},
\]
and therefore
\[
b
\qquad\text{if and only if}\qquad
\theta_{m+1}.
\]
This proves that every such $b$ is either false everywhere on $\Cone(W)$ or
else equivalent there to a single threshold $\theta_n$.
The comparability assertion follows immediately, since the thresholds are
linearly ordered by implication and the everywhere-false sentence implies every
sentence.
\end{proof}

We shall use the standard \(\SFourThree\) button criterion over
\(\SFourTwo\): if, at a world and for the relevant substitution language,
there are no two independent weak buttons, then the \(\SFourThree\) axiom is
valid there. Equivalently, any failure of \(\axiomf{.3}\) yields two
independent weak buttons; see the control-statement discussion in
\cite[Section~4]{WolSets}.

\begin{lemma}
\label{lem:S43-sentential-emb}
Every world in $\LOemb$ validates $\SFourThree$ for sentential
substitutions.
\end{lemma}

\begin{proof}
Suppose, towards a contradiction, that $b_0,b_1$ are independent sentential
weak buttons at a world $W$.
By Theorem~\ref{thm:S42-emb-all}, $\SFourTwo$ is valid at $W$ for
sentential substitutions, so Proposition~\ref{prop:button-purification} shows that both
$b_i$ are buttons.
Set $\widehat b_i=\necessary b_i$.
Then each $\widehat b_i$ is a pure sentential button, and independence is
preserved by this replacement: whenever a further world realizes a pattern of
pushed original buttons, exactly the corresponding $\widehat b_i$ hold there,
because $\widehat b_i$ says precisely that $b_i$ has been pushed.
Thus, after replacing $b_i$ by $\widehat b_i$ and dropping the hats, we may
assume that $b_0,b_1$ are independent pure sentential buttons.
In particular, each $b_i$ is upward absolute on the cone above $W$.
Since each original $b_i$ was a weak button, each pure version is possible
above $W$, so case~\textup{(i)} of Lemma~\ref{lem:upward-absolute-sentences}
cannot occur.
Therefore there are $n_0,n_1$ such that
$b_i\leftrightarrow \theta_{n_i}$ throughout the cone.
But the thresholds $\theta_n$ are linearly ordered by implication, so $b_0$
and $b_1$ cannot be independent, a contradiction.
Hence there is no pair of independent sentential weak buttons, and the logic
$\SFourThree$ follows.
\end{proof}

\begin{theorem}
\label{thm:upper-S5-emb-sent}
For every world $W$ in $\LOemb$ and every intermediate language
$\Lord\subseteq \mathcal L\subseteq \Lmord$,
\[
\Valsent_{\LOemb}(W,\mathcal L)\subseteq \SFive.
\]
\end{theorem}

\begin{proof}
Since enlarging the sentence-substitution language can only shrink the set of
valid propositional modal formulas, it is enough to treat the case
$\mathcal L=\Lord$.
By Proposition~\ref{prop:upper-bounds}(1), it suffices to produce arbitrarily
long finite dials.
Fix $N\ge 3$.
For $n\ge 2$, let $\mathrm{Disc}_n$ be the sentence saying that there is a
discrete interval of size $n$, namely an interval with exactly $n$ points.
Set
\[
d_1:=\neg \mathrm{Disc}_2,
\qquad
 d_n:=\mathrm{Disc}_n\wedge \neg\mathrm{Disc}_{n+1}
 \ (2\le n<N),
\qquad
d_{\ge N}:=\mathrm{Disc}_N.
\]
Exactly one of these sentences holds in any linear order.  Indeed, the
sentences $\mathrm{Disc}_n$ are nested: $\mathrm{Disc}_{n+1}$ implies
$\mathrm{Disc}_n$, since an $(n+1)$-point discrete interval contains an
$n$-point discrete subinterval.  Thus this dial records not how many discrete
intervals of a given size occur, but the largest finite size, below the cutoff
$N$, of a discrete interval occurring anywhere in the order.
Moreover, from any world we may embed into a world with no adjacent pair and
no endpoints, thereby destroying the old finite discrete intervals, and then
prepend or append a finite discrete interval of any prescribed size.  Since the
dense part has no endpoint at the cut, no larger finite discrete interval is
created across the cut.  This realizes every dial value.
\end{proof}

\par\smallskip\noindent
We next isolate the finite embedding worlds.  Let $\Phi_2$ be the propositional modal formula
\begin{align*}
\Phi_2 := {} &
\necessary\Bigl((q\vee p_0\vee p_1)\wedge \neg(p_0\wedge p_1)\wedge
\neg(p_0\wedge q)\wedge \neg(p_1\wedge q)\Bigr) \\
& \wedge \necessary\Bigl((p_0\to \possible p_1)\wedge (p_1\to \possible p_0)\Bigr) \\
& \wedge \necessary\Bigl((p_0\to \possible q)\wedge (p_1\to \possible q)\Bigr) \\
& \wedge \necessary\Bigl(q\to (\neg\possible p_0\wedge \neg\possible p_1)\Bigr).
\end{align*}
Set
\[
\chi_2:=\Phi_2\to \neg p_0,
\qquad
\Sfourthreecap:=\SFourThree+\chi_2.
\]
For $n,m\ge 1$, let $C_{n,m}$ be the finite preorder obtained from a chain of
$n$ singleton clusters
$c_0<\cdots<c_{n-1}$
by capping it with an $m$-cluster
$\{t_0,\dots,t_{m-1}\}$
above $c_{n-1}$.
These are the nondegenerate finite capped chains used below.
\par\smallskip

We use the standard local finite-frame characterization of $\SFourThree$:
for finite preorders, $\SFourThree$ holds exactly when every generated cone
has linearly ordered cluster quotient. Equivalently, a finite rooted preorder
validates $\SFourThree$ if and only if its cluster quotient is a finite
linear order; see, for example, \cite{BRV,CZ}.

\begin{proposition}
\label{prop:chi2}
A finite preorder validates $\chi_2$ if and only if every non-maximal cluster is
a singleton.
Consequently, a finite rooted preorder validates
$\Sfourthreecap$ if and only if it is a finite capped chain, allowing the
possibly degenerate case of no lower singleton clusters.
\end{proposition}

\begin{proof}
Assume first that some non-maximal cluster $C$ contains two distinct worlds
$c_0\neq c_1$.
Let $q$ be true exactly strictly above $C$, let $p_1$ be true exactly at $c_1$,
and let $p_0$ be true at the remaining worlds of the cluster $C$.
Then $\Phi_2$ holds at $c_0$, while $p_0$ also holds at $c_0$.
So $\chi_2$ fails.

Conversely, suppose every non-maximal cluster is a singleton and that
$\Phi_2\wedge p_0$ holds at some world $u$.
Consider the set of worlds above $u$ satisfying $p_0\vee p_1$, and let $a$ be
one in a maximal cluster meeting this set.
That cluster cannot be maximal. Indeed, $\Phi_2$ forces a $q$-successor
above $a$, and that successor cannot lie in the same cluster: if it did, the
fourth boxed conjunct of $\Phi_2$ would make it unable to see any $p_0$- or
$p_1$-world, although $a$ itself is such a world in the same cluster.
Hence the $q$-successor lies in a strictly higher cluster, and $a$ is itself
a singleton.
If $a\satisfies p_0$, then $\Phi_2$ forces a $p_1$-successor above $a$;
if $a\satisfies p_1$, it forces a $p_0$-successor above $a$.
Either way we contradict maximality.
Thus $\chi_2$ holds.

For the consequence, let $F$ be a finite rooted preorder.
If $F\satisfies \Sfourthreecap$, then in particular
$F\satisfies \SFourThree$, so the rooted $\SFourThree$ characterization
shows that the cluster quotient of $F$ is a finite linear order. It therefore
has a unique maximal, or top, cluster. By the first part, every non-top cluster
is a singleton. Hence $F$ is a finite capped chain in the possibly degenerate
sense above.
Conversely, every finite capped chain has linear cluster quotient and only
singleton non-top clusters, so the same characterization yields
$F\satisfies \SFourThree$, and the first part yields $F\satisfies \chi_2$.
Therefore $F\satisfies \Sfourthreecap$.
\end{proof}

\begin{corollary}
\label{cor:capped-chains-characterize}
The propositional modal theory of the frames $C_{n,m}$ with $n\ge 1$ and
$m\ge 2$ is exactly $\Sfourthreecap$.
\end{corollary}

\begin{proof}
Every frame $C_{n,m}$ is a finite capped chain, so
Proposition~\ref{prop:chi2} yields
$C_{n,m}\satisfies \Sfourthreecap$.
Hence
$\Sfourthreecap$
is contained in the theory of
\[
\mathcal C=\{C_{n,m}:n\ge 1,\ m\ge 2\}.
\]
Conversely, let
$\varphi\notin \Sfourthreecap$.
By Bull's finite model property theorem for normal extensions of
$\SFourThree$ \cite{BullFMP}; see also \cite{CZ}, there is a finite Kripke model
$(F,V)$ and a world $w_0$ such that $F$ validates $\Sfourthreecap$ and
\[
(F,V,w_0)\not\satisfies \varphi.
\]
Replace $F$ by the generated subframe $F_{w_0}$ and restrict $V$ to it.
By generated-subframe preservation, the formula $\varphi$ still fails at
$w_0$.
Moreover, generated subframes preserve validity: given any valuation on
$F_{w_0}$, extend it arbitrarily to a valuation on $F$ and apply
generated-subframe preservation at every world of $F_{w_0}$.
Hence $F_{w_0}$ still validates $\Sfourthreecap$.
The generated frame $F_{w_0}$ is finite and rooted. By
Proposition~\ref{prop:chi2}, it is a possibly degenerate finite capped chain.
If the top cluster has size at least $2$ and there is at least one lower
singleton cluster, then
$F_{w_0}\cong C_{n,m}$
for some $n\ge 1$ and $m\ge 2$, and $\varphi$ already fails on a frame in
$\mathcal C$.

There remains the degenerate case in which $F_{w_0}$ is a single complete
cluster of size $m\ge 2$.  In that case take $C_{1,m}$, and define a surjective
bounded morphism
\[
\pi:C_{1,m}\to F_{w_0}
\]
by sending the root singleton of $C_{1,m}$ to the distinguished point $w_0$ and
mapping the top $m$-cluster bijectively onto $F_{w_0}$.  Pulling back $V$ along
$\pi$ gives a valuation on $C_{1,m}$ under which $\varphi$ fails at the root.

Suppose next that the top cluster has size $1$ and that $F_{w_0}$ has more than
one world.
Form a new frame $F'$ by duplicating the unique top world of $F_{w_0}$ and
leaving the rest of the frame unchanged.
Then
$F'\cong C_{n,2}$
for some $n\ge 1$.
The collapse map
$\pi:F'\to F_{w_0}$
is a bounded morphism.
By bounded-morphism invariance, pulling back $V$ along $\pi$ yields a
valuation $V'$ on $F'$ with
\[
(F',V',w_0)\not\satisfies \varphi.
\]

Finally, if $F_{w_0}$ has exactly one world, let
$F''=C_{1,2}$,
and let
$\pi:F''\to F_{w_0}$
be the constant map.
This is again a bounded morphism.
By bounded-morphism invariance, pulling back $V$ along $\pi$ yields a
valuation $V''$ on $F''$ such that
$\varphi$ fails at the designated root of $C_{1,2}$.
Thus $\varphi$ fails on some frame in $\mathcal C$, and the theory of
$\mathcal C$ is exactly
$\Sfourthreecap$.
\end{proof}

\begin{lemma}
\label{lem:chi2-sentential-emb}
Every world in the category of linear orders and order-embeddings validates
$\chi_2$ under sentential substitutions from any intermediate language
$\Lord\subseteq \mathcal L\subseteq \Lmord$.
\end{lemma}

\begin{proof}
Let $U$ be any world of $\LOemb$, and let $A_0,A_1,B$ be the sentence
substitutions for $p_0,p_1,q$, respectively.  Suppose, towards a contradiction,
that $U\satisfies A_0\wedge\Phi_2$.  Since $U\satisfies A_0$ and the third boxed
conjunct of $\Phi_2$ holds throughout the cone above $U$, we have
$U\satisfies\possible B$.

The first and fourth boxed conjuncts make $B$ upward absolute on the cone above
$U$.  Indeed, the first boxed conjunct makes $B$, $A_0$, and $A_1$ exhaustive
and pairwise disjoint throughout that cone, while the fourth boxed conjunct
prevents a $B$-world from accessing an $A_0$- or $A_1$-world.  Since $B$ is
possible above $U$, Lemma~\ref{lem:upward-absolute-sentences} gives an $n$ such
that $B$ holds if and only if $\theta_n$ holds throughout the cone above $U$.
But $U\satisfies A_0$ and $A_0$ is disjoint from $B$, so $U\not\satisfies B$ and hence
$U$ is finite of size less than $n$.

Now consider the worlds $X\ge U$ satisfying $A_0\vee A_1$.  Each such $X$ is
not a $B$-world, and hence is finite of size less than $n$.  Choose such an $X$
of maximal size, with $X\satisfies A_i$ for some $i\in\{0,1\}$.  The second boxed
conjunct of $\Phi_2$ gives an extension $Y\ge X$ with $Y\satisfies A_{1-i}$.  Since
embeddings do not decrease size and $Y$ is again one of the finitely bounded
$A_0\vee A_1$-worlds, maximality gives $|Y|=|X|$.  Thus the embedding
$X\emb Y$ is an isomorphism of finite linear orders.  By the renaming lemma,
parameter-free sentences are invariant under this isomorphism, so
$Y\satisfies A_i$ as well.  This contradicts the disjointness of $A_0$ and $A_1$.
Therefore no such $U$ exists, and $\chi_2$ is valid at every world under
sentential substitutions.
\end{proof}

\begin{lemma}
\label{lem:capped-chain-labeling}
Let $W$ be a finite linear order of size $k$.
For every $n\ge 1$ and $m\ge 2$, the capped chain frame $C_{n,m}$ admits a
labeling above $W$ using only $\Lord$-sentences.
\end{lemma}

\begin{proof}
Suppose $|W|=k$.
Fix $n\ge 1$ and $m\ge 2$, and write
\[
C_{n,m}=\{c_0<\cdots<c_{n-1}\}\cup\{t_0,\dots,t_{m-1}\},
\]
with initial node $c_0$.

For the chain nodes, set
\[
\Phi_{c_i}:=\theta_{k+i}\wedge \neg\theta_{k+i+1}
\qquad (0\le i<n).
\]

For the cap, use the $m$-dial from the proof of
Theorem~\ref{thm:upper-S5-emb-sent}. Namely,
\[
d_1:=\neg\mathrm{Disc}_2,
\qquad
d_r:=\mathrm{Disc}_r\wedge \neg\mathrm{Disc}_{r+1}
\ (2\le r<m),
\qquad
d_{\ge m}:=\mathrm{Disc}_m.
\]
Define
\[
\Phi_{t_j}:=
\begin{cases}
\theta_{k+n}\wedge d_{j+1}, & 0\le j<m-1,\\[1mm]
\theta_{k+n}\wedge d_{\ge m}, & j=m-1.
\end{cases}
\]

We verify the labeling conditions.

First, since $|W|=k$, we have
\[
W\satisfies \theta_k\wedge\neg\theta_{k+1},
\]
so $W\satisfies \Phi_{c_0}$.

Second, let $U\ge W$.
Then $|U|\ge k$.
If $|U|=k+i$ for some $i<n$, then
$U\satisfies \Phi_{c_i}$,
and $U$ satisfies no other chain label and no cap label.
Otherwise $|U|\ge k+n$, so
$U\satisfies \theta_{k+n}$.
For fixed $m$, exactly one of
$d_1,\dots,d_{m-1},d_{\ge m}$
holds in $U$.
Thus $U$ satisfies exactly one of the cap labels
$\Phi_{t_0},\dots,\Phi_{t_{m-1}}$.

Third, we check that accessibility agrees with the frame order.

Suppose first that $U\satisfies \Phi_{c_i}$ and $i\le i'<n$.
Then $U$ is the finite linear order of size $k+i$, which embeds into the unique
finite linear order of size $k+i'$.
Hence some world above $U$ satisfies $\Phi_{c_{i'}}$.

Now suppose $U\satisfies \Phi_{c_i}$ and $v$ is any cap node.
By the dial construction from Theorem~\ref{thm:upper-S5-emb-sent}, we may first
embed $U$ into a world $U_1\ge U$ having no adjacent pairs and no endpoints, and
then prepend a finite discrete interval $D_r$ of any prescribed size $r\ge 1$.
Choosing $r$ so that the resulting dial value is the one used in the label of $v$,
we obtain an extension of $U$ satisfying $\theta_{k+n}$ together with that dial
value, and hence satisfying $\Phi_v$.

Now suppose $U\satisfies \Phi_{c_i}$ and $i'<i$. Every world satisfying
$\Phi_{c_{i'}}$ has size $k+i'<k+i=|U|$, so there is no embedding from $U$ into
any world satisfying $\Phi_{c_{i'}}$.
Hence $U\not\satisfies \possible\Phi_{c_{i'}}$.

If $U\satisfies \Phi_{t_j}$, then $|U|\ge k+n>k+i$ for every $i<n$, so there is no
embedding from $U$ into any world satisfying $\Phi_{c_i}$.
Thus $U\not\satisfies \possible\Phi_{c_i}$ for all $i<n$.

Finally, if $U\satisfies \Phi_{t_j}$ and $j'<m$, then by the same dial construction,
applied above $U$, there is an extension of $U$ satisfying the cap label
$\Phi_{t_{j'}}$.
Therefore every cap world accesses every cap world.

Combining these cases, for every node $v$ of $C_{n,m}$ we have:
a world $U\ge W$ satisfies $\possible\Phi_v$ if and only if the node already labeling $U$ is
below $v$ in $C_{n,m}$.
So the modalities align exactly with the frame order, and the displayed formulas
form a labeling of $C_{n,m}$ above $W$ using only $\Lord$-sentences.
\end{proof}

\begin{theorem}
\label{thm:finite-emb-sentential}
If $W$ is a finite world in $\LOemb$, then its sentential propositional modal
validities are exactly $\Sfourthreecap$.
More precisely, for every intermediate language
$\Lord\subseteq \mathcal L\subseteq \Lmord$,
\[
\Valsent_{\LOemb}(W,\mathcal L)=\Sfourthreecap.
\]
\end{theorem}

\begin{proof}
We prove both inclusions.

For the lower bound, Lemma~\ref{lem:S43-sentential-emb} gives
$\SFourThree$ throughout the cone, and
Lemma~\ref{lem:chi2-sentential-emb} gives $\chi_2$ throughout the cone for the
same sentential substitution languages.  By the cone lemma, the normal modal
logic generated by these principles is valid at $W$, giving the lower inclusion
$\Sfourthreecap\subseteq \Valsent_{\LOemb}(W,\mathcal L)$.

For the upper bound, Lemma~\ref{lem:capped-chain-labeling} yields labelings of
every frame $C_{n,m}$ with $n\ge 1$ and $m\ge 2$ above $W$ using
$\Lord$-sentences.
By the labeling lemma, every propositional modal formula not valid on all such
frames fails at $W$ under a sentential substitution.
By Corollary~\ref{cor:capped-chains-characterize}, the valid formulas on all
frames $C_{n,m}$ with $n\ge 1$ and $m\ge 2$ are exactly
$\Sfourthreecap$.
\end{proof}

\begin{theorem}
\label{thm:infinite-emb-sentential}
Let $W$ be a world in $\LOemb$.
Then, for every intermediate language
$\Lord\subseteq \mathcal L\subseteq \Lmord$,
\[
\Valsent_{\LOemb}(W,\mathcal L)=\SFive
\]
if and only if $W$ is infinite.
\end{theorem}

\begin{proof}
Suppose first that the sentential validities of $W$ are exactly
$\SFive$.
For each $n$, the sentence $\theta_n$ is upward absolute and possible above $W$.
Hence
$W\satisfies \possible\necessary\theta_n$.
Since $\SFive$ is valid at $W$, $\axiomf{5}$ yields
$W\satisfies \theta_n$.
Thus $W$ is infinite.

Conversely, assume $W$ is infinite.
Let $b$ be a sentential weak button at $W$.
By Theorem~\ref{thm:S42-emb-all}, $\SFourTwo$ is valid for sentence
instances, so $b$ is a button.
Replacing $b$ by its pure version, we may assume it is upward absolute.
Since $b$ is a weak button, some world above $W$ satisfies it, so
case~\textup{(i)} of Lemma~\ref{lem:upward-absolute-sentences} cannot occur.
Hence
$b\leftrightarrow \theta_n$ on the cone for some $n$.
Since $W$ is infinite, $W\satisfies \theta_n$, and therefore $W\satisfies b$.
Thus every sentential instance of $\axiomf{5}$ holds at $W$.
Every world above $W$ is also infinite, so the same argument applies on the
whole cone.
By the cone lemma, the normal closure of
$\SFour+\axiomf{5}=\SFive$ is valid at $W$.
The reverse inclusion is Theorem~\ref{thm:upper-S5-emb-sent}.
\end{proof}

\par\smallskip\noindent
We now turn to the full parameter language.  The following elementary finite-gap obstruction will be used.  Let
$P=\prod_{c\in C}L_c$ be a finite product of finite chains, ordered
coordinatewise, and let $B_0,\ldots,B_{n-1}$ be upward-closed subsets of $P$.
For $p\in P$, write
\[
I(p)=\{i<n:p\in B_i\}.
\]
Say that the family is necessarily independent above $p_0$ if $I(p_0)=\emptyset$
and, whenever $p\ge p_0$ and $I(p)\subseteq S\subseteq n$, there is $q\ge p$
with $I(q)=S$.

\begin{lemma}\label{lem:finite-gap-independent-buttons}
If $B_0,\ldots,B_{n-1}$ are necessarily independent upward-closed subsets of
$P=\prod_{c\in C}L_c$ above some point $p_0$, then $n\le |C|$.
\end{lemma}

\begin{proof}
Choose a maximal $p\ge p_0$ with $I(p)=\emptyset$.  For each $i<n$, necessary
independence gives $q_i\ge p$ with $I(q_i)=\{i\}$.  Take a saturated chain from
$p$ to $q_i$.  Since $p$ is maximal among points with empty pattern, the first
step of this chain already leaves the empty region.  That first step changes a
single coordinate, say $c_i$, and reaches a point $r_i\le q_i$.  Upwardness and
$I(q_i)=\{i\}$ imply $I(r_i)\subseteq\{i\}$, while maximality of $p$ gives
$I(r_i)\ne\emptyset$; hence $I(r_i)=\{i\}$.
If $c_i=c_j$, then the first step above $p$ in that coordinate is the same
point, so $r_i=r_j$ and therefore $\{i\}=I(r_i)=I(r_j)=\{j\}$.  Thus
$i\mapsto c_i$ is injective, and $n\le |C|$.
\end{proof}

\begin{theorem}
\label{thm:emb-S42-iff-adjacent}
Let $W$ be a world in $\LOemb$.
Then
\[
\Val_{\LOemb}(W,\Lmord(W))=\SFourTwo
\]
if and only if $W$ has infinitely many adjacent pairs.
\end{theorem}

\begin{proof}
The proof is guided by the following picture.  Infinitely many adjacent pairs
supply independent buttons, while a separate infinite side of the order supplies
finite dials.  Conversely, if there are only finitely many adjacent pairs, then
only finitely many gap coordinates can support pure buttons.

Assume first that $W$ has infinitely many adjacent pairs.
Choose $c\in W$ so that one side of $c$ contains infinitely many adjacent pairs.
From that side choose, for each prescribed finite size, pairwise endpoint-disjoint
adjacent pairs.  This is possible since, after choosing finitely many adjacent
pairs, only finitely many further adjacent pairs are ruled out by sharing an
endpoint with one of them: every point is an endpoint of at most two adjacent
pairs.  Since the side contains infinitely many adjacent pairs, one can continue
until the prescribed finite size is reached.  If this side is $(c,\infty)$, take
the buttons from adjacent pairs above $c$ and relativize the finite dials from
Theorem~\ref{thm:upper-S5-emb-sent} to the initial segment $(-\infty,c)$.  If
the infinite side is $(-\infty,c)$, take the buttons from adjacent pairs below
$c$ and relativize those dials to the final segment $(c,\infty)$.
For each chosen adjacent pair $a<b$ on that side, let
\[
\beta_{a,b}:=\exists z\,(a<z<b).
\]
Each $\beta_{a,b}$ is a pure button: once the interval $(a,b)$ becomes nonempty
it remains nonempty under further embeddings.
Buttons supported on endpoint-disjoint adjacent pairs are independent.  The dials are
supported on the opposite side of $c$ and are therefore independent from the
buttons.  By Theorem~\ref{thm:S42-emb-all} and
Proposition~\ref{prop:upper-bounds}(2), we conclude
$\Val_{\LOemb}(W,\Lmord(W))=\SFourTwo$.

Conversely, suppose
$\Val_{\LOemb}(W,\Lmord(W))=\SFourTwo$.
By the sharpness clause in Proposition~\ref{prop:upper-bounds}(2),
$W$ admits arbitrarily large finite independent families of unpushed pure buttons
with parameters, together with arbitrarily long finite dials.
Assume towards a contradiction that $W$ has only finitely many adjacent pairs,
and hence only finitely many relevant insertion gaps.
If $W$ is empty, let $C=\{I_*\}$ consist of one formal whole-order gap.  If
$W$ is nonempty, let $C$ be the finite set consisting of the open gaps
$(x,y)$ for adjacent pairs $x<y$ of $W$, together with the end-gaps
$(-\infty,\min W)$ if $W$ has a minimum and $(\max W,\infty)$ if
$W$ has a maximum.

Fix an independent family of pure buttons
$\langle b_i(\bar a_i):i<n\rangle$ over $W$.
We view each accessible world as an extension of $W$ along the chosen
embedding, replacing it by an isomorphic copy over $W$ when necessary.
Thus the gap-size vector below is always computed relative to the fixed copy
of $W$.
For each world $U\ge W$ and each gap $I\in C$, let $I^U$ be the corresponding
interval in $U$; in the empty case, put $I_*^U=U$.  Set
\[
\vec\ell(U)=(\ell_I(U))_{I\in C}\in (\omega\cup\{\infty\})^C,
\]
where $\ell_I(U)=|I^U|$ if finite and $\ell_I(U)=\infty$ otherwise.
If $U\emb V$, then $\vec\ell(U)\le \vec\ell(V)$ coordinatewise.

Working in the language naming the elements of $W$, let
\[
T_W:=\Th(\mathrm{LO})\cup \diag(W),
\]
where $\diag(W)$ is the full atomic diagram of $W$.
By modality elimination for embeddings, each pure button $b_i(\bar a_i)$ is
equivalent over $\LOemb$ to a non-modal formula of $\LredA{W}$.
Because $b_i(\bar a_i)$ is pure, it is upward absolute on $\Cone(W)$.
Hence whenever
\[
U\subseteq V\satisfies T_W
\]
and
$U\satisfies b_i(\bar a_i)$,
the inclusion
$U\emb V$
shows that
$V\satisfies b_i(\bar a_i)$.
Since $T_W$ is universal, the Lo\'s--Tarski preservation theorem, applied over
$T_W$, yields an existential $\LredA{W}$-formula
\[
\exists \bar x\,\psi_i(\bar x,\bar a_i)
\]
equivalent to $b_i(\bar a_i)$ over $T_W$, with $\psi_i$ quantifier-free.

Write $\psi_i$ in disjunctive normal form and complete each satisfiable
conjunction to a full order-pattern over the parameters $\bar a_i$ and the
finitely many named elements of $W$ occurring in that conjunction.  We record
why each completed pattern is controlled by finitely many thresholds on the
coordinates in $C$.

Fix one completed pattern, and let $A\subseteq W$ be the finite set of named
points of $W$ occurring in it.  Consider the cells determined by $A$: if
$A=\emptyset$ there is one whole-order cell, while otherwise these are the two
open endpoint cuts and the open intervals between consecutive points of $A$.  We include
such a cell even when its intersection with $W$ is empty, since inserted
witnesses may still lie there.  Variables identified with members of $A$ impose
no insertion requirements.  In one such cell $J$, suppose that
$h$ variable-equivalence
classes of the pattern are required to lie strictly inside $J$.  If
$J\cap W$ has at least $h$ elements, then these classes can be realized by
points of $W$ itself, in the order prescribed by the pattern, so $J$ imposes no
condition on inserted points.  If $J\cap W$ has fewer than $h$ elements, then
$J\cap W$ is finite.  List all points of $J\cap W$ in increasing order, together
with the endpoints of $J$ from $A$ when such endpoints are present.  The
complementary gaps inside $J$ are then precisely gaps belonging to $C$: gaps
coming from adjacent pairs of $W$, endpoint gaps when the corresponding endpoint
exists, and in the case $W=\emptyset$ the single formal gap $I_*$.  To realize
the $h$ ordered classes in $J$, one chooses which of them are assigned to the
finitely many available points of $J\cap W$ and how many are assigned to each of
these complementary gaps.  There are only finitely many such choices, and each
choice is realized in a world $U\ge W$ exactly when each involved gap $I\in C$
contains at least the prescribed finite number of inserted points.  Conversely,
when those lower bounds are met, the chosen assignment realizes the required
part of the order-pattern in $J$.

Taking the finite product over the cells and the finite union over the possible
assignments gives, for this completed pattern, a finite family of threshold
vectors in $\omega^C$.  The pattern is realized in $U$ if and only if
$\vec\ell(U)$ is coordinatewise above one of those vectors.  Since there are
only finitely many completed patterns in the disjunctive normal form, there are
finitely many threshold vectors
$t_i^{(1)},\dots,t_i^{(r_i)}\in\omega^C$
such that for all
$U\ge W$,
\[
U\satisfies b_i(\bar a_i)
\qquad\text{if and only if}\qquad
\vec\ell(U)\ge t_i^{(j)}\text{ for some }j\le r_i,
\]
where $\infty$ is understood as $\ge n$ for every $n<\omega$.
In particular, the set of vectors forcing $b_i$ is upward closed in
$(\omega\cup\{\infty\})^C$.

Let $N$ be at least every finite coordinate appearing in any of the finitely
many threshold vectors $t_i^{(j)}$.  Truncate vectors coordinatewise by
\[
\widehat v_I=\min(v_I,N),
\]
where $\min(\infty,N)=N$, and let $P_N=\{0,\ldots,N\}^C$.  The threshold
description induces upward-closed subsets $B_i\subseteq P_N$ such that, for
every $U\ge W$,
\[
U\satisfies b_i(\bar a_i)
\qquad\text{if and only if}\qquad
\widehat{\vec\ell(U)}\in B_i.
\]
These subsets are necessarily independent above the zero vector.  Indeed, given
$p\in P_N$, realize $p$ by an extension $U_p\ge W$ inserting exactly $p_I$ new
points in each gap $I\in C$; when $W=\emptyset$, this simply means that $U_p$
has $p_{I_*}$ points.  The value $N$ merely means that we insert $N$ points,
which is enough for all thresholds under consideration.  If $I(p)$ is the set of
buttons true at $U_p$ and $I(p)\subseteq S\subseteq n$, independence of the
buttons in the cone above $W$ gives an extension $V\ge U_p$ in which exactly the
buttons in $S$ are true.  Then
$\widehat{\vec\ell(V)}\ge p$ and lies in exactly the sets $B_i$ with $i\in S$.
By Lemma~\ref{lem:finite-gap-independent-buttons}, we have $n\le |C|$.
Since the sharpness assumption provides such independent families for
arbitrarily large $n$, this is impossible when $C$ is finite.  Thus $W$ must
have infinitely many adjacent pairs.
\end{proof}

\begin{theorem}
\label{thm:emb-S5-iff-DLO}
For a world $W$ in $\LOemb$,
\[
\Val_{\LOemb}(W,\Lmord(W))=\SFive
\]
if and only if $W$ is a nonempty dense linear order without endpoints.
\end{theorem}

\begin{proof}
The upper inclusion always holds: full parameter-language validities are contained
in the sentential validities obtained by using parameter-free $\Lord$-sentences,
and Theorem~\ref{thm:upper-S5-emb-sent} puts those sentential validities inside
$\SFive$.  For the lower inclusion, use
Theorem~\ref{thm:elimination-emb}: the category $\LOemb$ admits modality
elimination.  By the existential-closedness criterion just recalled, a world
validates $\SFive$ in its full parameter language exactly when it is
existentially closed as a linear order.  Since empty orders are allowed in the ambient category, the empty order is not
existentially closed: it embeds into a singleton, which realizes the existential
sentence $\exists x\,x=x$.  The existentially closed worlds are therefore
exactly the nonempty dense linear orders without endpoints.
\end{proof}

Thus, under embeddings, nonempty dense endpoint-free orders are precisely the worlds at
which every possibly necessary parameter assertion is already true.  Orders
with infinitely many adjacent pairs sit lower in the modal hierarchy: their
parameter logic is exactly directedness, namely $\SFourTwo$.

\section{The modal theory of linear orders and monotone maps}

For monotone maps the lower bound is again $\SFourTwo$, but the upper behavior
is different. Monotone maps may collapse intervals, and this makes sentences much
less sensitive to the internal shape of a nonempty order. Sententially, every
nonempty world validates $\SFive$. Parameters, however, can mark points whose
relative behavior survives collapse, and so the parameter language distinguishes
singleton worlds from nonempty infinite worlds. Let $\LOmon$ denote the Kripke
category of linear orders and monotone maps.

\begin{theorem}
\label{thm:S42-mon-all}
Every world in $\LOmon$ validates $\SFourTwo$ for arbitrary substitution
instances, even with parameters.
\end{theorem}

\begin{proof}
Every span in $\LOmon$ is amalgamable because the singleton order is terminal.
Indeed, if
$L_0\xleftarrow{}L\xrightarrow{}L_1$
is a span, let $\{s\}$ be the singleton order.
There is a unique monotone map from each $L_i$ to $\{s\}$, and uniqueness makes
the square commute.
So Theorem~\ref{thm:S42-lower-bound} yields $\SFourTwo$.
\end{proof}

\begin{theorem}
\label{thm:upper-S5-mon-all}
If $W$ is a nonempty world in $\LOmon$, then for every intermediate language
\[
\Lord\subseteq \mathcal L\subseteq \Lmord(W),
\]
one has
\[
\Val_{\LOmon}(W,\mathcal L)\subseteq \SFive.
\]
\end{theorem}

\begin{proof}
The sentential dial from the proof of
Theorem~\ref{thm:upper-S5-emb-sent} is available in $\LOmon$ as well.
Indeed, every nonempty linear order admits a monotone map to every nonempty
target order: choose any point of the target and take the constant map.
In particular, each of the target orders used there is accessible from $W$.
Since the dial values already lie in the base language
$\Lord\subseteq \mathcal L$, Proposition~\ref{prop:upper-bounds}(1) yields the
claimed inclusion $\Val_{\LOmon}(W,\mathcal L)\subseteq \SFive$.
\end{proof}

\begin{theorem}
\label{thm:nonempty-mon-sentential}
If $W$ is a nonempty world in $\LOmon$, then for every intermediate language
$\Lord\subseteq \mathcal L\subseteq \Lmord$,
\[
\Valsent_{\LOmon}(W,\mathcal L)=\SFive.
\]
\end{theorem}

\begin{proof}
Every nonempty linear order admits a monotone map to every other nonempty linear
order, so every nonempty world is weakly terminal in $\LOmon$.
Hence Theorem~\ref{thm:terminal-S5}(1) gives the lower bound
$\SFive$ sententially.
For the reverse inclusion, use the same sentential dial as in the proof of
Theorem~\ref{thm:upper-S5-mon-all}. Since its values already lie in
$\Lord\subseteq \mathcal L$, Proposition~\ref{prop:upper-bounds}(1), applied
to the sentence-substitution language, gives the reverse inclusion
$\Valsent_{\LOmon}(W,\mathcal L)\subseteq \SFive$.
\end{proof}

\begin{theorem}
\label{thm:singleton-mon-formulaic}
If $W$ is a singleton world in $\LOmon$, then
\[
\Val_{\LOmon}(W,\Lmord(W))=\SFive.
\]
\end{theorem}

\begin{proof}
A singleton world is terminal in $\LOmon$, because there is exactly one
monotone map into it.
So Theorem~\ref{thm:terminal-S5}(2) gives the lower bound
$\SFive$.
The reverse inclusion is Theorem~\ref{thm:upper-S5-mon-all}.
\end{proof}

\begin{theorem}
\label{thm:mon-S42-iff-infinite}
Let $W$ be a nonempty world in $\LOmon$.
Then
\[
\Val_{\LOmon}(W,\Lmord(W))=\SFourTwo
\]
if and only if $W$ is infinite.
\end{theorem}

\begin{proof}
Assume first that $W$ is infinite.
Choose $c\in W$ so that one side $J$ of $c$ is infinite, and let $K$ be the
opposite side.  Fix $N<\omega$ and choose pairwise distinct parameters
$u_1<v_1<\cdots<u_N<v_N$ inside $J$; if $J=(c,\infty)$ this means
$c<u_1<v_1<\cdots<u_N<v_N$, while if $J=(-\infty,c)$ it means
$u_1<v_1<\cdots<u_N<v_N<c$.
Let
\[
\rho(\bar u,\bar v):=\bigvee_{j\neq k}
(u_j=u_k\vee u_j=v_k\vee v_j=v_k),
\]
and define
\[
b_i:=(u_i=v_i)\vee \rho(\bar u,\bar v)
\qquad (1\le i\le N).
\]
Each $b_i$ is an unpushed pure button.  Indeed, if $b_i$ is false in an
accessible world, then the transported parameters are still separated and
$\rho$ is false there; collapsing the convex interval between the images of
$u_i$ and $v_i$ pushes $b_i$, while any equality among transported parameters is
preserved by further monotone maps.  The family is independent in the necessary
sense.  Let $M$ be any accessible world and let
$P=\{i:M\satisfies b_i\}$.  If $M\satisfies\rho$, then $P=\{1,\dots,N\}$ and there is
nothing to prove.  Otherwise the transported intervals
$[u_i^M,v_i^M]$ remain linearly ordered and pairwise separated except possibly
at endpoints already forced by monotonicity.  Given any target set
$S\supseteq P$, collapse exactly the intervals $[u_i^M,v_i^M]$ with
$i\in S\setminus P$.  No equality $u_j=v_j$ with $j\notin S$ is introduced, and
no instance of $\rho$ is introduced, so exactly the buttons in $S$ are pushed.

On the opposite side $K$, and for arbitrary dial length, use the dial recording
the maximum size of a discrete interval in $K$, with a catch-all last value.
Even if $K$ is finite or empty in the current world, every finite value of this
dial is reachable. Indeed, from any accessible world, replace the strict side $K$
of the transported image of $c$---that is, the open initial segment below it or
the open final segment above it---by a finite chain of the desired length and
keep $c$ and the other side fixed.  This sets the dial value while preserving all button states,
since the button parameters lie on side $J$.  Conversely, pushing any of the
buttons collapses only intervals on side $J$ and therefore preserves the dial
value.  Thus arbitrarily long finite dials are independent from the buttons.
By Theorem~\ref{thm:S42-mon-all} and Proposition~\ref{prop:upper-bounds}(2),
we obtain exact validity of $\SFourTwo$.

Conversely, suppose
$\Val_{\LOmon}(W,\Lmord(W))=\SFourTwo$.
By the sharpness clause in Proposition~\ref{prop:upper-bounds}(2),
$W$ admits arbitrarily large finite independent families of unpushed pure buttons
together with arbitrarily long finite dials.
Assume towards a contradiction that $W$ is finite.

Let $f:W\to M$ be any monotone map.
Since $W$ is finite, the image $f(W)$ is finite.
Let
\[
d_0<\cdots<d_{m-1}
\]
list the distinct points of $f(W)$ in increasing order, let $M'$ be the
induced finite suborder on $\{d_0,\dots,d_{m-1}\}$, let
$q:W\to M'$
be the corestriction of $f$, and let
$e:M'\to M$
be the inclusion, so that
$f=e\circ q$.
Because $M'$ is finite, there is also a monotone retraction
$g:M\to M'$
defined by sending $x\in M$ to the largest $d_i\le x$ if such exists, and to
$d_0$ otherwise.
Then
\[
g\circ e=\mathrm{id}_{M'}
\qquad\text{and}\qquad
g\circ f=q.
\]

Now let $b(\bar x)$ be any pure button at $W$, and let $\bar a$ be parameters
from $W$.
Since $b$ is pure on the cone above $W$, its truth is preserved along every
arrow in that cone.
Applying this to the arrows
$g:M\to M'$
and
$e:M'\to M$,
we obtain $M\satisfies b[f(\bar a)]$ if and only if $M'\satisfies b[q(\bar a)]$. Therefore, for any fixed finite family of pure buttons, the realized
button-pattern in an accessible world depends only on the corestriction
$q:W\to M'$, equivalently on the convex partition of $W$ into the fibres of
the map.

A finite linear order has only finitely many convex partitions.
Hence only finitely many button-patterns can be realized above $W$.
But an independent family of $N$ pure buttons requires all $2^N$ possible
patterns to be realizable.
This is impossible for arbitrarily large $N$.
Therefore $W$ must be infinite.
\end{proof}

\section{The modal theory of linear orders and condensations}

Condensations are the opposite extreme. A condensation can collapse large convex
pieces of an order, and the previous section showed that this makes scatteredness
visible to the modal language. We turn next to $\LOcond$, the category of
nonempty linear orders and condensations. The natural substitution language here
is the full modal language with parameters, because the exact logic is controlled
by formulas using the modal predicate $\mathrm{Scatt}_{<}(x)$. The result is a useful
contrast with embeddings. Instead of adjacent pairs providing the main buttons,
the decisive persistent possibility is whether an initial segment is
non-scattered. In particular, the exact logic below is obtained at every
non-scattered world; no global condensation onto $\Q$ is required.

\begin{theorem}
\label{thm:S421-cond-all}
Every nonempty world in $\LOcond$ validates $\SFourTwoOne$ for arbitrary
substitution instances with parameters.
\end{theorem}

\begin{proof}
Let $W$ be nonempty and let $\{s\}$ be a singleton world accessible from $W$.
Every world accessible from $\{s\}$ is again a singleton, and every condensation
out of a singleton is unique.
Thus every span in the cone above $W$ amalgamates over the singleton,
so $\SFourTwo$ is valid by Theorem~\ref{thm:S42-lower-bound}.

It remains to verify $\axiomf{.1}$, that is, the McKinsey axiom
$\necessary\possible p\to \possible\necessary p$.
Suppose $W\satisfies \necessary\possible\varphi[\bar a]$.
Let $h:W\onto\{s\}$ be the unique condensation.
Then
\[
\{s\}\satisfies \possible\varphi[h(\bar a)].
\]
So choose a condensation
\[
j:\{s\}\onto S
\]
with
\[
S\satisfies \varphi[j(h(\bar a))].
\]
Every world accessible from a singleton is again a singleton.
Now let
\[
k:\{s\}\onto T
\]
be any condensation.
Then $T$ is a singleton, and there is a unique isomorphism
$\pi:S\cong T$.
By the renaming lemma,
\[
T\satisfies \varphi[\pi(j(h(\bar a)))].
\]
Since all maps out of a singleton are unique, the tuples
$\pi(j(h(\bar a)))$ and $k(h(\bar a))$ agree.
Hence
\[
T\satisfies \varphi[k(h(\bar a))].
\]
Because $k$ was arbitrary, we conclude
\[
\{s\}\satisfies \necessary\varphi[h(\bar a)].
\]
Thus $W\satisfies \possible\necessary\varphi[\bar a]$, witnessed by $h$.

The same argument applies at every world in the cone above $W$.  Indeed, if
$f:W\onto U$ is a condensation and the parameters are transported to
$f(\bar a)$, then $U$ is again nonempty and hence condenses to a singleton;
from that singleton all further condensations are unique.  The preceding
renaming argument therefore proves the McKinsey implication at $U$ for the
transported parameters.  Thus the McKinsey axiom holds throughout
$\Cone(W)$.  By the cone lemma, $\SFourTwoOne$ is valid at $W$.
\end{proof}

\begin{lemma}
\label{lem:sigma-button-condensation}
Let $L$ be a nonempty linear order and let $c\in L$.  If the initial segment
$(-\infty,c)_L$ is non-scattered, then
\[
\sigma:=\mathrm{Scatt}_{<}(c)
\]
is an unpushed pure button at $L$ in $\LOcond$.
\end{lemma}

\begin{proof}
The hypothesis says exactly that $L\satisfies\neg\sigma$, so the button is
unpushed.

To see that it is a button, let $j:L\onto M$ be any condensation and let
$k:M\onto\{s\}$ be the unique condensation to a singleton.  The initial segment
below the unique point of $\{s\}$ is empty, hence scattered.  Therefore
$\{s\}\satisfies \mathrm{Scatt}_{<}[k(j(c))]$, and this truth is necessary there,
since every further accessible world from a singleton is again a singleton.
Thus every world accessible from $L$ can reach a world in which $\sigma$ is
necessary.

Purity follows from Corollary~\ref{cor:scattered-below}.  If $\sigma$ holds at a
world $M$ and fails at a further condensation $M\onto M'$, then the definition of
$\mathrm{Scatt}_{<}$ would already witness the failure of $\sigma$ at $M$.
\end{proof}

\begin{lemma}
\label{lem:sigma-dials-and-buttons-nonscattered-condensation}
Let $L$ be a non-scattered nonempty linear order.  Then there is a point
$c\in L$ such that, putting
\[
\sigma:=\mathrm{Scatt}_{<}(c),
\]
$\sigma$ is an unpushed pure button at $L$, and below $\sigma$ there are
arbitrarily long finite $\sigma$-dials and arbitrarily large finite independent
families of unpushed pure $\sigma$-buttons, independent from those dials.
\end{lemma}

\begin{proof}
Since $L$ is non-scattered, fix an embedded copy
$e:\Q\emb L$.  Choose $q\in\Q$ and put $c=e(q)$.  Then
$e((-\infty,q)_{\Q})$ is a copy of $\Q$ below $c$, so
$(-\infty,c)_L$ is non-scattered.  Lemma~\ref{lem:sigma-button-condensation}
therefore shows that $\sigma$ is an unpushed pure button at $L$.

We next build the $\sigma$-dials.  Let $\mathrm{Adj}(x,y)$ abbreviate adjacency,
namely $x<y\wedge \neg\exists z\,(x<z\wedge z<y)$.  For $n\ge 1$, let
$\mathrm{Chain}^{<}_n(x)$ assert that there are
$y_0<\cdots<y_{n-1}<x$ with
$\mathrm{Adj}(y_i,y_{i+1})$ for every $i<n-1$.  Thus all points of the
chain lie in the open initial segment $(-\infty,x)$.  Define
\[
\delta_n(x):=\mathrm{Chain}^{<}_n(x)\wedge
\neg\mathrm{Chain}^{<}_{n+1}(x).
\]
For $N\ge 2$, put
\[
d_i(x):=\delta_{i+1}(x)\qquad (i<N),
\qquad
d_{\ge N}(x):=\neg\bigl(d_0(x)\vee\cdots\vee d_{N-1}(x)\bigr).
\]
For any parameter $x$, the assertions $\delta_n[x]$ are pairwise incompatible:
if $m>n$, then $\delta_m[x]$ implies $\mathrm{Chain}^{<}_{n+1}[x]$, while
$\delta_n[x]$ implies its negation.  Hence exactly one of
$d_0[x],\dots,d_{N-1}[x],d_{\ge N}[x]$ holds at every world.  It remains to show
that every desired dial value can be reached while staying below $\sigma$.

Let $M$ be any world below $\sigma$ accessible from $L$, and write $c_M$ for the
image of $c$ in $M$.  Thus $M\satisfies\neg\sigma$, equivalently
$M\satisfies\neg\mathrm{Scatt}_{<}[c_M]$.  We first pass to a convenient hub below
$c_M$ and then quotient that hub so as to create a prescribed finite adjacency
chain below the image of $c_M$, without changing the final segment
$[c_M,\infty)_M$ apart from retagging it in the target.

Since $M\satisfies\neg\sigma$, Corollary~\ref{cor:scattered-below} yields a
condensation $u:M\onto U$ such that
\[
U\satisfies \mathrm{Max}[u(c_M)]\wedge\neg\mathrm{Scatt}.
\]
Because $\neg\mathrm{Scatt}$ is equivalent to $\possible\mathrm{Dense}$, there
is a condensation $v:U\onto D$ with $D\satisfies\mathrm{Dense}$.  Put
$r=v\circ u:M\onto D$ and let $c_D=r(c_M)$.  Then $c_D$ is a maximum of $D$.
Let
\[
D^{-}=\{d\in D:d<c_D\}
\qquad\text{and}\qquad
I=r^{-1}(D^{-})\subseteq M.
\]
Since $D^{-}$ is an initial segment of $D$ and $r$ is monotone, the set $I$ is
an initial segment of $M$, contained in $\{x\in M:x<c_M\}$.  Define a linear
order
\[
H=(D^{-}\times\{0\})\cup (\{x\in M:x\ge c_M\}\times\{1\}),
\]
ordered as a tagged sum, and define $h:M\to H$ by
\[
h(x)=
\begin{cases}
 (r(x),0),&x\in I,\\
 (c_M,1),&x<c_M\text{ and }x\notin I,\\
 (x,1),&x\ge c_M.
\end{cases}
\]
Then $h$ is monotone and surjective, hence a condensation.  Moreover,
$h(c_M)=(c_M,1)$, and the segment below it in $H$ is exactly
$D^{-}\times\{0\}$.  Since $D\satisfies\mathrm{Dense}$, this segment is a
nontrivial dense linear order with a minimum and no maximum.

Now fix $n\ge 1$.  Let $H^-:=\{x\in H:x<h(c_M)\}$.  Choose
$t_0<\cdots<t_{n-1}$ in $H^-$ with $t_0$ above the minimum of $H^-$.  Let
$B=\{x\in H^-:x<t_0\}$.  The order $B$ is again non-scattered, dense, and has
no maximum.  Let $p_0<\cdots<p_{n-1}$ be a finite chain of $n$ new points, and
let
\[
K_n:=B+p_0+\cdots+p_{n-1}+
\{x\in H:x\ge h(c_M)\},
\]
where the last summand keeps its order from $H$.  There is a natural
condensation $g_n:H\onto K_n$: it is the identity on $B$ and on the final
summand $\{x\in H:x\ge h(c_M)\}$, it collapses the convex intervals
$[t_i,t_{i+1})$ to $p_i$ for $i<n-1$, and it collapses
$[t_{n-1},\infty)_{H^-}$ to $p_{n-1}$.  The segment below $g_n(h(c_M))$ is
$B+p_0+\cdots+p_{n-1}$.  The points $p_0<\cdots<p_{n-1}$ form an
adjacency-chain of length $n$, and there is no adjacency-chain of length $n+1$ below
$g_n(h(c_M))$, since $B$ is dense with no maximum.  Thus
$K_n\satisfies\delta_n[g_n(h(c_M))]$.  Also, $B$ is non-scattered, so
$K_n\satisfies\neg\mathrm{Scatt}_{<}[g_n(h(c_M))]$, and the construction remains
below $\sigma$.

Taking $n=i+1$ realizes $d_i$ for each $i<N$, and taking $n=N+1$ realizes
$d_{\ge N}$.  Hence the displayed list is a $\sigma$-dial below $\sigma$.

It remains to produce independent $\sigma$-buttons.  Choose rationals
\[
q<r_0<s_0<r_1<s_1<\cdots<r_{m-1}<s_{m-1},
\]
and set $a_i=e(r_i)$ and $b_i=e(s_i)$.  Then
\[
c<a_0<b_0<a_1<b_1<\cdots<a_{m-1}<b_{m-1}
\]
in $L$.  Define
\[
\beta_i:=\sigma\vee(a_i=b_i)\qquad (i<m).
\]
The same argument as above shows that each $\beta_i$ is an unpushed pure
$\sigma$-button.  Indeed, if $M$ is a world below $\sigma$ accessible from $L$
and the images of $a_i$ and $b_i$ are still distinct, then collapsing the closed
convex interval $[a_i^M,b_i^M]$ pushes $\beta_i$.  This interval lies at or above
$c_M$; even if its left endpoint is $c_M$, the open initial segment below the
image of $c$ is unchanged, hence still non-scattered.  Once $\beta_i$ holds, it
is necessary: either $\sigma$ holds, and $\sigma$ is pure, or the equality
$a_i=b_i$ holds, and equalities are preserved by further condensations.

The family is independent below $\sigma$ in the necessary sense.  Let $M$ be any
world below $\sigma$ accessible from $L$, and let the images of the parameters
be denoted by the same letters.  Monotonicity gives
\[
c\le a_0\le b_0\le\cdots\le a_{m-1}\le b_{m-1}.
\]
Below $\sigma$, the button $\beta_i$ is true exactly when $a_i=b_i$.  If
$P=\{i<m:a_i=b_i\}$ is the current pushed set and $S\supseteq P$, collapse the
closed intervals $[a_i,b_i]$ with $i\in S\setminus P$.  These convex intervals
are linearly ordered, possibly meeting only at endpoints, so the quotient is a
linear order and the quotient map is a condensation.  No equality $a_j=b_j$
with $j\notin S$ is introduced, and all collapses occur at or above $c$, so the
open initial segment below $c$ is unchanged.
The resulting world therefore remains below $\sigma$ and realizes exactly the
chosen target pattern $S$.

Finally, the $\sigma$-dial and the family of $\sigma$-buttons are independent
below $\sigma$.  The dial constructions change only the open initial segment
below the transported parameter $c$, and so preserve the equalities $a_i=b_i$;
pushing the buttons collapses intervals at or above $c$, and so preserves the
finite adjacency-chain information in that open initial segment.  This completes
the proof.
\end{proof}

\begin{theorem}
\label{thm:nonscattered-condensation-exact-S421}
If $L$ is a non-scattered nonempty linear order, then
\[
\Val_{\LOcond}(L,\Lmord(L))=\SFourTwoOne.
\]
Consequently, the propositional modal validities of the category $\LOcond$, for
substitutions from the full modal language with parameters, are exactly
$\SFourTwoOne$.
\end{theorem}

\begin{proof}
The lower bound is Theorem~\ref{thm:S421-cond-all}.  For the upper bound, apply
Lemma~\ref{lem:sigma-dials-and-buttons-nonscattered-condensation} to obtain a
point $c\in L$ such that $\sigma=\mathrm{Scatt}_{<}(c)$ is an unpushed pure button,
and below $\sigma$ there are arbitrarily long finite $\sigma$-dials together
with arbitrarily large independent families of unpushed pure $\sigma$-buttons.
Hence Proposition~\ref{prop:upper-bounds}(4) yields the upper inclusion
$\Val_{\LOcond}(L,\Lmord(L))\subseteq \SFourTwoOne$.
The final assertion follows because $\Q$ is non-scattered, and the lower
bound holds at every nonempty world.
\end{proof}

\section{The modal theory of linear orders and end-extensions}

Finally, let us consider the category of linear orders and end-extension
embeddings. This returns to the embedding side of the picture, but the modal
behavior is quite different. A finite world can only grow by adding a tail above
its old top, and this one-sided growth is flexible enough to label arbitrary
finite pretrees. The proof is constructive. Above a fixed finite base order, we
build a family of finite tails that code dials, branch choices, and finally a
railyard labeling.

Let $\LOend$ be the subcategory of $\LOemb$ with the same worlds but whose
morphisms are required to be end-extension embeddings, that is, embeddings whose
images are initial segments.  If $W$ is a linear order and $T$ another linear
order, we write $W+T$ for the ordered sum obtained by placing every point of
$T$ above every point of $W$.  The inclusion $W\emb W+T$ is then an
end-extension.

We next classify the finite worlds of $\LOend$.  The purpose of the construction
is to make a finite propositional frame visible in the possible finite tails
above the old order.  It has three stages.  First we define a dial by looking at
the shape of a final finite tail.  Then we code branch choices by inserting
carefully chosen sums of finite chains.  Finally we combine these devices to
label a finite pretree.

Fix a finite base order $L$.  All formulas in the construction are interpreted
relative to this fixed base order.  The formulas $P_n$ below mark points which
terminate a successor chain of length $n$ whose first point lies at or above the
embedded copy of $L$.  The first point is not required to be the old top of $L$;
this local definition is needed later, when finite successor blocks are placed
after dense buffers.  Thus $\mathrm{Empty}_n$, $\mathrm{Min}_n$, and
$\mathrm{NoMin}_n$ allow us to detect whether such $P_n$-points are absent, have
a least representative, or occur without a least representative.

Let $\eta$ be a dense order without endpoints, and let $C_r$ be the finite chain
of size $r$.
Let $\mathrm{Succ}(x,y)$ express that $y$ is the immediate successor of $x$.
If $L\neq\varnothing$ and $|L|=n$, let $\mathrm{Top}_L(x)$ say that $x$ has exactly
$n-1$ predecessors; in any end-extension this defines the top point of the embedded
copy of $L$.
Set
\[
\mathrm{Anch}_L(x):=\exists y\,(\mathrm{Top}_L(y)\wedge y\le x),
\qquad
\mathrm{Above}_L(x):=\exists y\,(\mathrm{Top}_L(y)\wedge y<x),
\]
with both formulas taken to be tautologies if $L=\varnothing$.
Define inductively
\[
P_1(z):=\exists y\,(\mathrm{Anch}_L(y)\wedge \mathrm{Succ}(y,z)),
\qquad
P_{n+1}(z):=\exists y\,(P_n(y)\wedge \mathrm{Succ}(y,z)).
\]
Finally, let
\begin{align*}
\mathrm{Empty}_n&:=\forall z\,\neg P_n(z),\\
\mathrm{Min}_n&:=\exists z\,(P_n(z)\wedge \forall w\,(P_n(w)\to z\le w)),\\
\mathrm{NoMin}_n&:=\bigl(\exists z\,P_n(z)\bigr)
\wedge \forall z\,(P_n(z)\to \exists w\,(P_n(w)\wedge w<z)).
\end{align*}
We shall also use the following elementary first-order abbreviations:
\[
\begin{aligned}
\mathrm{HasPred}(x)&:=\exists y\,\mathrm{Succ}(y,x),\\
\mathrm{NoPred}(x)&:=\neg\mathrm{HasPred}(x),\\
\mathrm{NoSucc}(x)&:=\neg\exists y\,\mathrm{Succ}(x,y),\\
\mathrm{Max}(x)&:=\neg\exists y\,(x<y).
\end{aligned}
\]

\begin{lemma}
\label{lem:end-dial}
For every $m>0$ there are sentential formulas $d_0,\dots,d_{m-1}$ forming a
dial above $L$ in $\LOend$.
Moreover, there is $N_{\mathrm{dial}}$ such that:
\begin{enumerate}[label=\textup{(\arabic*)}]
    \item every end-extension of $L$ satisfies exactly one of
    $d_0,\dots,d_{m-1}$;
    \item for each $j<m$ there is a linear order $T_j$ such that
    $W+T_j\satisfies d_j$ for every end-extension $W$ of $L$;
    \item appending $T_j$ creates no new $P_n$-points for any
    $n\ge N_{\mathrm{dial}}$.
\end{enumerate}
\end{lemma}

\begin{proof}
Choose pairwise distinct integers $D_1,\dots,D_{m-1}\ge 1$.
Set $D_{\max}=0$ if $m=1$, and otherwise
$D_{\max}=\max\{D_1,\dots,D_{m-1}\}$.
Let
$N_{\mathrm{dial}}=D_{\max}+2$.
For $r\ge 1$, let $\mathrm{Exactly}_r(b,u)$ be a standard first-order formula
asserting that exactly $r$ many points $x$ satisfy $b<x\le u$, and put
\[
\begin{aligned}
\mathrm{FinalBlock}_r(b,u):={}&
\mathrm{Above}_L(b)\wedge \mathrm{Max}(u)\wedge \mathrm{NoPred}(b)\\
&\wedge \forall x\,\bigl((b<x\wedge x\le u)\to \mathrm{HasPred}(x)\bigr)\\
&\wedge \mathrm{Exactly}_r(b,u).
\end{aligned}
\]
For $j\ge 1$, define
\[
    d_j:=\exists b\exists u\,\mathrm{FinalBlock}_{D_j}(b,u).
\]
Let
$d_0:=\neg(d_1\vee\cdots\vee d_{m-1})$,
or simply $d_0:=\top$ when $m=1$.

The formulas $d_1,\dots,d_{m-1}$ are mutually exclusive.  Indeed, suppose
\[
\mathrm{FinalBlock}_r(b,u)\qquad\text{and}\qquad
\mathrm{FinalBlock}_s(b',u')
\]
both hold.  Then $u=u'$ by uniqueness of the maximum.  If $b<b'$, then
$b'<u$ because
$s\ge 1$, so the final-block condition for $b$ implies that $b'$ has an
immediate predecessor, contradicting $\mathrm{NoPred}(b')$.  The case
$b'<b$ is symmetric, hence $b=b'$, and then $r=s$ by the exact-cardinality
clause.  Since the integers $D_j$ are pairwise distinct, at most one of
$d_1,\dots,d_{m-1}$ can hold.  Hence exactly one of
$d_0,\dots,d_{m-1}$ holds in every end-extension of $L$, establishing
clause~\textup{(1)}.

For clause~\textup{(2)}, let $T_0=\eta$, and for $j\ge 1$ let
\[
T_j=\eta+C_{D_j+1}.
\]
Then $W+T_0=W+\eta$ has no maximum, so it satisfies $d_0$.
For $j\ge 1$, the order
$W+T_j=W+(\eta+C_{D_j+1})$
has a maximum, namely the top of the final finite chain.
The dense buffer $\eta$ ensures that the bottom of that final chain has no
immediate predecessor, so $W+T_j\satisfies d_j$.

For clause~\textup{(3)}, each $T_j$ begins with a dense block $\eta$.
Therefore no element of $W$ has an immediate successor in the appended tail,
and successor chains cannot cross from $W$ into the new part.
If $j=0$, then $T_0=\eta$ has no successor steps at all, so it creates no new
$P_n$-points.
If $j\ge 1$, any successor chain inside
$\eta+C_{D_j+1}$
is contained in the finite block $C_{D_j+1}$, and therefore has at most $D_j$
successor steps.
So appending $T_j$ creates no new $P_n$-points for
$n\ge D_j+2$.
With
$N_{\mathrm{dial}}=D_{\max}+2$,
this proves clause~\textup{(3)} for all $j<m$.
\end{proof}

\begin{lemma}
\label{lem:end-branch-coding}
Let $k\ge 1$.
There are sentential formulas $\mathrm{Branch}_{n,a}$ for $n\ge 1$ and
$a<k$, and an integer $B_{\max}$, such that:
\begin{enumerate}[label=\textup{(\arabic*)}]
    \item if $\neg \mathrm{Empty}_n$ holds, then exactly one
    $\mathrm{Branch}_{n,a}$ holds;
    \item branch values persist under end-extensions;
    \item if $n\ge \max\{1,B_{\max}\}$, $W$ is an end-extension of the fixed
    finite base order $L$, and $W\satisfies\mathrm{Empty}_n$, then for every $a<k$
    there is a tail $T(n,a)$ such that
    \[
        W+T(n,a)\satisfies
        \mathrm{Branch}_{n,a}\wedge \mathrm{Empty}_{n+1}.
    \]
\end{enumerate}
\end{lemma}

\begin{proof}
If $k=1$, set
\[
\mathrm{Branch}_{n,0}:=\neg\mathrm{Empty}_n
\qquad\text{and}\qquad
B_{\max}:=0.
\]
Then \textup{(1)} and \textup{(2)} are immediate. For \textup{(3)}, let
$n\ge1$ and let $W$ be an end-extension of the fixed finite base order $L$ with
$W\satisfies\mathrm{Empty}_n$.  Take
\[
T(n,0)=\sum_{q\in\eta} C_{n+1}.
\]
Then $W+T(n,0)\satisfies \neg\mathrm{Empty}_n\wedge \mathrm{Empty}_{n+1}$: the
tail $\sum_{q\in\eta} C_{n+1}$ creates $P_n$-points, but because the indexing
copy of $\eta$ has no least element and the dense indexing prevents successor
chains from crossing between finite chains, there is no least $P_n$-point and
no $P_{n+1}$-point. Hence
$W+T(n,0)\satisfies \mathrm{Branch}_{n,0}\wedge \mathrm{Empty}_{n+1}$, as
required.
So assume $k\ge 2$.
Choose pairwise distinct integers $B_0,\dots,B_{k-2}\ge 2$, and let
$B_{\max}=\max\{B_0,\dots,B_{k-2}\}$.
For $B\ge 2$, define the following first-order formulas:
\[
\begin{aligned}
\mathrm{MaxSuccBelow}(y,b):={}&
 b<y\wedge \mathrm{HasPred}(b)\\
&\wedge\forall c\,\bigl((b<c\wedge c<y)\to\neg\mathrm{HasPred}(c)\bigr),\\[2mm]
\mathrm{Chain}_B(b):={}&
\exists c_0\cdots\exists c_{B-1}\Bigl(
  c_{B-1}=b\wedge \mathrm{NoPred}(c_0)\\
&\hspace{35mm}\wedge\bigwedge_{i<B-1}\mathrm{Succ}(c_i,c_{i+1})
\Bigr).
\end{aligned}
\]
Thus $\mathrm{MaxSuccBelow}(y,b)$ says that $b$ is the greatest point below $y$
which has an immediate predecessor, and $\mathrm{Chain}_B(b)$ says that $b$ is
the top of a successor chain of size $B$ whose bottom has no immediate
predecessor.  For $a\le k-2$, put
\[
\mathrm{Code}_a(y):=\exists b\,\bigl(
\mathrm{MaxSuccBelow}(y,b)\wedge \mathrm{Chain}_{B_a}(b)
\wedge \mathrm{NoSucc}(b)\bigr).
\]
Also let
\[
\mathrm{Min}P_n(x):=P_n(x)\wedge\forall z\,(P_n(z)\to x\le z).
\]
For $a\le k-2$, define $\mathrm{Branch}_{n,a}$ to say explicitly that the least
$P_n$-point is reached by a successor chain whose bottom carries code $a$:
\[
\exists x\,\exists y_0\cdots\exists y_n\Bigl(
   \mathrm{Min}P_n(x)
   \wedge y_0<\cdots<y_n=x
   \wedge \bigwedge_{i<n}\mathrm{Succ}(y_i,y_{i+1})
   \wedge \mathrm{Code}_a(y_0)
\Bigr).
\]
Finally let
\[
\mathrm{Branch}_{n,k-1}:=\neg\mathrm{Empty}_n
\wedge\bigwedge_{a<k-1}\neg\mathrm{Branch}_{n,a}.
\]
For statement~\textup{(1)}, if the $P_n$-points have no least element, then all
special branches are false and the residual branch is true.  If they have a
least element, then the relevant predecessor chain below it is unique; moreover
$\mathrm{MaxSuccBelow}$ chooses at most one coding point and the distinct sizes
$B_a$ make the codes mutually exclusive.  Hence exactly one branch value holds.
Statement~\textup{(2)} follows because end-extensions add points only above the
old order.  They cannot add a smaller $P_n$-point, and all configurations used
by the special codes lie strictly below the least $P_n$-point, so the special
and residual branch values are preserved.
For statement (3), if $a\le k-2$, take
\[
T(n,a)=\eta+C_{B_a}+\eta+C_{n+1};
\]
if $a=k-1$, take
\[
T(n,k-1)=\sum_{q\in\eta} C_{n+1}.
\]
The dense buffers keep successor chains from crossing between blocks.  In the
special case $a\le k-2$, the assumption $n\ge B_{\max}$ ensures that the coding
block $C_{B_a}$ creates no $P_n$-point, while the final block $C_{n+1}$ creates
a least $P_n$-point and no $P_{n+1}$-point.  The bottom of this final block sees,
below the intervening dense buffer, the coded block of size $B_a$, so
$\mathrm{Branch}_{n,a}$ holds.  In the residual case $a=k-1$, the tail
$\sum_{q\in\eta}C_{n+1}$ creates $P_n$-points with no least one and still creates
no $P_{n+1}$-point.  Thus the indicated tails create exactly the required
branch value and preserve $\mathrm{Empty}_{n+1}$.
\end{proof}

\begin{lemma}
\label{lem:end-tree-address}
Let $T$ be a finite rooted tree every non-leaf of which has exactly $k$ children,
these children being enumerated as $t^\frown 0,\dots,t^\frown(k-1)$.
For $t\in T$, let $\operatorname{addr}(t)\in k^{<\omega}$ be its address.
Fix $N\ge\max\{1,B_{\max}\}$ and define
\[
\tau_t:=
\bigwedge_{i<|\operatorname{addr}(t)|}
\mathrm{Branch}_{N+i,\operatorname{addr}(t)(i)}
\ \wedge\
\begin{cases}
\mathrm{Empty}_{N+|\operatorname{addr}(t)|},&\text{if }t\text{ is not a leaf},\\
\top,&\text{if }t\text{ is a leaf}.
\end{cases}
\]
Then every end-extension of $L$ satisfies exactly one $\tau_t$.
Moreover:
\begin{enumerate}[label=\textup{(\arabic*)}]
    \item if $W\emb W'$ and $W\satisfies \tau_t$ and $W'\satisfies \tau_u$, then
    $t\le_T u$;
    \item if $W\satisfies \tau_t$ and $t\le_T u$, then there is an end-extension
    $W\emb W'$ with $W'\satisfies \tau_u$.
\end{enumerate}
\end{lemma}

\begin{proof}
Existence and uniqueness are obtained by reading the branch values beginning at
level $N$ and proceeding down the tree until one meets the first empty level.
Persistence of the branch predicates gives (1).
For (2), if $u$ extends $t$ by one more digit $a$, apply the third clause of
Lemma~\ref{lem:end-branch-coding} at the first empty level to force
$\mathrm{Branch}_{n,a}$ while preserving emptiness at the next level.
Iterating finitely many times reaches $u$.
\end{proof}

\begin{proposition}
\label{prop:railyard-end}
Every finite world in $\LOend$ admits a railyard labeling of every finite regular
pretree.
\end{proposition}

\begin{proof}
Fix a regular finite pretree $F$ of type $(k,m)$ and let $T$ be its quotient
tree.
Let $d_0,\dots,d_{m-1}$ and $N_{\mathrm{dial}}$ be as in
Lemma~\ref{lem:end-dial}, and let the branch-coding formulas and $B_{\max}$ be
as in Lemma~\ref{lem:end-branch-coding}.
Choose
\[
N=\max\{B_{\max},N_{\mathrm{dial}}\}+2.
\]
For each $t\in T$, define $\tau_t$ using this $N$ as in
Lemma~\ref{lem:end-tree-address}.
Enumerate the cluster over each $t\in T$ as
$t_0,\dots,t_{m-1}$,
and define
\[
r_{t_j}:=\tau_t\wedge d_j.
\]

By Lemma~\ref{lem:end-tree-address}, every world accessible from the base world
satisfies exactly one $\tau_t$, and by clause~\textup{(1)} of
Lemma~\ref{lem:end-dial}, exactly one of the dial values $d_0,\dots,d_{m-1}$
holds there.
Hence every accessible world satisfies exactly one label $r_{t_j}$.
If $r$ is the root of $T$, then in the base world there are no $P_n$-points for
any $n\ge 1$, so $\mathrm{Empty}_N$ holds and therefore
$\tau_r$ holds.
Also the base world satisfies $d_0$.
Hence the initial world carries the designated root label attached to the
first point of the root cluster.

Now suppose
$W\emb W'$, $W\satisfies r_{t_i}$, and $W'\satisfies r_{u_j}$.
Then
$W\satisfies \tau_t$
and
$W'\satisfies \tau_u$,
so Lemma~\ref{lem:end-tree-address}\textup{(1)} gives
$t\le_T u$.
Consequently
$t_i\le_F u_j$.

Conversely, suppose
$W\satisfies r_{t_i}$
and
$t_i\le_F u_j$.
Then
$t\le_T u$.
First apply Lemma~\ref{lem:end-tree-address}\textup{(2)} to obtain an
end-extension
$W\emb W_1$
with
$W_1\satisfies \tau_u$.
Since
$N\ge N_{\mathrm{dial}}$,
clause~\textup{(2)} of Lemma~\ref{lem:end-dial} gives a further
end-extension
$W_1\emb W_2$
with
$W_2\satisfies d_j$.
By clause~\textup{(3)} of Lemma~\ref{lem:end-dial}, this second step creates no
new $P_n$-points for any $n\ge N$, and therefore does not affect the truth of
any $\mathrm{Branch}_{n,a}$ or $\mathrm{Empty}_n$ appearing in $\tau_u$.
Hence
$W_2\satisfies \tau_u\wedge d_j=r_{u_j}$.
So the accessibility relation on the labels is exactly the pretree order, and
we have the required railyard labeling.
\end{proof}

\begin{theorem}
\label{thm:end-finite-S4}
If $L$ is a finite world in $\LOend$, then its propositional modal validities are
exactly $\SFour$.
More precisely, for every intermediate language
$\Lord\subseteq \mathcal L\subseteq \Lmord(L)$,
\[
\Val_{\LOend}(L,\mathcal L)=\SFour.
\]
\end{theorem}

\begin{proof}
The lower bound $\SFour$ holds in every Kripke category.
Conversely, Proposition~\ref{prop:upper-bounds}(3) reduces the upper bound to
railyard labelings of finite regular pretrees using assertions from
$\mathcal L$. Proposition~\ref{prop:railyard-end} supplies such labelings by
parameter-free order-modal sentences, hence by assertions belonging to every
intermediate language $\Lord\subseteq\mathcal L\subseteq\Lmord(L)$.
Hence $\Val_{\LOend}(L,\mathcal L)\subseteq \SFour$, and equality follows.
In modal terms, finite end-extension worlds have no
unavoidable validities beyond reflexivity and transitivity, because their finite
tails can simulate any finite $\SFour$-frame.
\end{proof}

\begin{remark}
By relativizing the railyard construction above a fixed parameter, one obtains
infinite linear orders whose formulaic validities in $\LOend$ are exactly
$\SFour$.
In particular, the rational line has exact formulaic logic $\SFour$ in the
end-extension category.
This argument does not settle the corresponding question for $\omega$.
\end{remark}

\begin{theorem}
\label{thm:end-extends-to-S5}
Every world in $\LOend$ end-extends to a world validating $\SFive$ with
parameters.
\end{theorem}

\begin{proof}
By Theorem~\ref{prop:chains-to-S5}, it suffices to show that every set-sized
chain in $\LOend$ admits a covering cocone.
In fact we shall build such a cocone explicitly.

Let
\[
\mathcal W=\langle W_\alpha,f_{\alpha\beta}:W_\alpha\emb W_\beta\rangle_{
\alpha\le\beta<\lambda}
\]
be a set-sized chain in $\LOend$, where each $f_{\alpha\beta}$ is an
end-extension embedding and $f_{\alpha\alpha}=\mathrm{id}_{W_\alpha}$.
Form the disjoint union
\[
X=\bigsqcup_{\alpha<\lambda}(\{\alpha\}\times W_\alpha).
\]
Define a relation $\sim$ on $X$ by
\[
(\alpha,x)\sim (\beta,y)
\quad\text{if and only if}\quad
\exists \gamma<\lambda\,
\bigl(\gamma\ge \alpha,\beta \ \wedge\
f_{\alpha\gamma}(x)=f_{\beta\gamma}(y)\bigr).
\]
This is an equivalence relation.
Reflexivity is witnessed by $\gamma=\alpha$, symmetry is immediate, and
transitivity follows from coherence of the chain:
if $(\alpha,x)\sim (\beta,y)$ via $\gamma_0$ and
$(\beta,y)\sim (\xi,z)$ via $\gamma_1$, then for
$\delta=\max\{\gamma_0,\gamma_1\}$ we have
\[
f_{\alpha\delta}(x)=f_{\beta\delta}(y)=f_{\xi\delta}(z).
\]

Let
\[
N=X/\!\sim,
\]
and write $[(\alpha,x)]$ for the $\sim$-class of $(\alpha,x)$.
Define a relation $<$ on $N$ by
\[
\begin{aligned}
[(\alpha,x)]<[(\beta,y)]
\quad\text{if and only if}\quad
&\exists \gamma<\lambda\text{ with }\gamma\ge \alpha,\beta\text{ and}\\
&f_{\alpha\gamma}(x)<f_{\beta\gamma}(y)\text{ in }W_\gamma.
\end{aligned}
\]

We first record the persistence property.
If
$\alpha,\beta\le \gamma\le \delta<\lambda$
and
$f_{\alpha\gamma}(x)<f_{\beta\gamma}(y)$,
then by coherence,
\[
f_{\alpha\delta}
=
f_{\gamma\delta}\circ f_{\alpha\gamma},
\qquad
f_{\beta\delta}
=
f_{\gamma\delta}\circ f_{\beta\gamma},
\]
and since $f_{\gamma\delta}$ is an order-embedding, we obtain
\[
f_{\alpha\delta}(x)<f_{\beta\delta}(y).
\]
The same argument shows that equality also persists upward in the chain.

Using this persistence, the relation $<$ is well-defined on $\sim$-classes.
Suppose
$(\alpha,x)\sim (\alpha',x')$ and $(\beta,y)\sim (\beta',y')$.
If
$[(\alpha,x)]<[(\beta,y)]$
is witnessed by $\gamma_0$, let $\gamma_1$ witness
$(\alpha,x)\sim (\alpha',x')$ and let $\gamma_2$ witness
$(\beta,y)\sim (\beta',y')$.
With
$\delta=\max\{\gamma_0,\gamma_1,\gamma_2\}$,
persistence gives
\[
f_{\alpha\delta}(x)<f_{\beta\delta}(y),
\qquad
f_{\alpha\delta}(x)=f_{\alpha'\delta}(x'),
\qquad
f_{\beta\delta}(y)=f_{\beta'\delta}(y'),
\]
and hence
$f_{\alpha'\delta}(x')<f_{\beta'\delta}(y')$.
So
$[(\alpha',x')]<[(\beta',y')]$,
and the reverse implication is symmetric.

Now $(N,<)$ is a linear order.
For trichotomy, fix
$[(\alpha,x)]$ and $[(\beta,y)]$, and let $\gamma=\max\{\alpha,\beta\}$.
In the linear order $W_\gamma$, exactly one of the following holds:
\[
f_{\alpha\gamma}(x)<f_{\beta\gamma}(y),\qquad
f_{\alpha\gamma}(x)=f_{\beta\gamma}(y),\qquad
f_{\beta\gamma}(y)<f_{\alpha\gamma}(x).
\]
These yield respectively
$[(\alpha,x)]<[(\beta,y)]$,
$[(\alpha,x)]=[(\beta,y)]$,
or
$[(\beta,y)]<[(\alpha,x)]$.
For transitivity, if
$[(\alpha,x)]<[(\beta,y)]$
is witnessed by $\gamma_0$ and
$[(\beta,y)]<[(\xi,z)]$
is witnessed by $\gamma_1$, then with
$\delta=\max\{\gamma_0,\gamma_1\}$,
persistence gives
\[
f_{\alpha\delta}(x)<f_{\beta\delta}(y)<f_{\xi\delta}(z),
\]
and hence
$[(\alpha,x)]<[(\xi,z)]$.

For each $\alpha<\lambda$, define
\[
i_\alpha:W_\alpha\to N,
\qquad
i_\alpha(x)=[(\alpha,x)].
\]
Each $i_\alpha$ is an order-embedding.
If $x<y$ in $W_\alpha$, then
$i_\alpha(x)<i_\alpha(y)$
witnessed at stage $\gamma=\alpha$.
Conversely, if
$i_\alpha(x)<i_\alpha(y)$
is witnessed by some $\gamma\ge\alpha$, then
$f_{\alpha\gamma}(x)<f_{\alpha\gamma}(y)$,
and because $f_{\alpha\gamma}$ is an embedding, it follows that $x<y$.

The family
$\langle i_\alpha\rangle_{\alpha<\lambda}$
is a cocone.
Indeed, for $\alpha\le\beta$ and $x\in W_\alpha$,
\[
i_\beta(f_{\alpha\beta}(x))
=
[(\beta,f_{\alpha\beta}(x))]
=
[(\alpha,x)]
=
i_\alpha(x),
\]
since $(\beta,f_{\alpha\beta}(x))\sim (\alpha,x)$ is witnessed by $\gamma=\beta$.

Each $i_\alpha$ is moreover an end-extension embedding.
Suppose
$[(\beta,y)]<i_\alpha(x)= [(\alpha,x)]$.
Choose $\gamma\ge \alpha,\beta$ witnessing this inequality, so
\[
f_{\beta\gamma}(y)<f_{\alpha\gamma}(x)
\]
in $W_\gamma$.
Because $f_{\alpha\gamma}[W_\alpha]$ is an initial segment of $W_\gamma$,
it follows that $f_{\beta\gamma}(y)$ already lies in that image.
So there is $z\in W_\alpha$ with
$f_{\alpha\gamma}(z)=f_{\beta\gamma}(y)$,
and therefore
\[
[(\beta,y)]=[(\alpha,z)]\in i_\alpha[W_\alpha].
\]
Hence $i_\alpha[W_\alpha]$ is an initial segment of $N$.

Finally, the cocone is covering, because every element of $N$ is of the form
$[(\alpha,x)]=i_\alpha(x)$ for some $\alpha<\lambda$ and $x\in W_\alpha$.
Thus every set-sized chain in $\LOend$ admits a covering cocone.
Theorem~\ref{prop:chains-to-S5} therefore yields a world validating
$\SFive$ with parameters above any given world of $\LOend$.
\end{proof}

\printbibliography[heading=mlobibliography]

\end{document}